\documentclass{amsart}

\usepackage{amssymb,latexsym}
\usepackage{amsmath,amsfonts,amsthm}
\usepackage{stmaryrd}
\usepackage{enumerate}
\usepackage[T1]{fontenc}
\usepackage[latin1]{inputenc}
\usepackage[english]{babel}
\usepackage{url}

\usepackage[all]{xy}
\CompileMatrices
\SelectTips{cm}{}

\def\cxymatrix#1{\xy*[c]\xybox{\xymatrix#1}\endxy}

\usepackage{graphicx}
\usepackage{caption}
\usepackage{subfig}
\usepackage{ifpdf}
\ifpdf
\fi

\theoremstyle{plain}
\newtheorem{thm}{Theorem}[section]
\newtheorem{lemma}[thm]{Lemma}
\newtheorem{prop}[thm]{Proposition}
\newtheorem{cor}[thm]{Corollary}
\theoremstyle{definition}
\newtheorem{defn}[thm]{Definition}
\newtheorem{ex}[thm]{Example}

\newtheorem*{ack}{Acknowledgements}

\theoremstyle{remark}
\newtheorem{remark}[thm]{Remark}
\numberwithin{equation}{section}

\DeclareMathOperator{\id}{id}

\DeclareMathOperator{\Li}{Li}
\DeclareMathOperator{\FT}{FT}

\DeclareMathOperator{\Image}{Im}

\DeclareMathOperator{\SL}{SL}
\DeclareMathOperator{\GL}{GL}
\DeclareMathOperator{\PSL}{PSL}

\DeclareMathOperator{\Tor}{Tor}
\DeclareMathOperator{\Ext}{Ext}
\DeclareMathOperator{\Ker}{Ker}
\DeclareMathOperator{\Aut}{Aut}

\DeclareMathOperator{\ord}{ord}

\DeclareMathOperator{\ind}{ind}
\DeclareMathOperator{\Gal}{Gal}

\newcommand{\R}{\mathbb R} 
\newcommand{\C}{\mathbb C}

\newcommand{\F}{\textnormal F}

\newcommand{\Z}{\mathbb Z}
\newcommand{\Q}{\mathbb Q}
\newcommand{\B}{\mathcal B}
\newcommand{\Pre}{\mathcal P}

\begin{document}
\title[The extended Bloch group and algebraic $K$--theory]{The extended Bloch
group and\\algebraic $K$--theory}
\author{Christian~K. Zickert}
\address{Department of Mathematics\\University of California, Berkeley, CA
94720-3840, USA}
\email{zickert@math.berkeley.edu}
\begin{abstract} We define an extended Bloch group for an arbitrary field
$F$, and show that this group is naturally isomorphic to~$K_3^{\ind}(F)$ if $F$ is a number field. 
This gives an explicit description of~$K_3^{\ind}(F)$ in terms of generators and relations. We give a concrete formula for the regulator, and derive concrete symbol expressions generating the torsion. As an application, we show that a hyperbolic $3$--manifold with finite volume and invariant trace field $k$ has a fundamental class in $K_3^{\ind}(k)\otimes\Z[\frac{1}{2}]$.
\end{abstract} 

\maketitle

\section{Introduction}
The extended Bloch group $\widehat\B(\C)$ was introduced by Walter Neumann~\cite{Neumann} in his computation of the Cheeger--Chern--Simons class related to $\PSL(2,\C)$.
It is a $\Q/\Z$--extension of the classical Bloch group $\B(\C)$, and was used by Neumann to give explicit simplicial formulas for the volume and Chern--Simons invariant of hyperbolic $3$--manifolds; see also Zickert~\cite{Zickert}.
There are two distinct versions of the extended Bloch group. One is isomorphic to $H_3(\PSL(2,\C)^\delta)$ and the other is isomorphic to $H_3(\SL(2,\C)^\delta)$. The $\delta$ denotes that the groups are regarded as discrete groups, and will from now on be omitted. For a discussion of the relationship between the two versions of the extended Bloch group, see Goette--Zickert~\cite{GZ}.

In Section~\ref{BhatE}, we define an extended Bloch group for an arbitrary field $F$. More precisely, we show that there is an extended Bloch group $\widehat\B_E(F)$ for each extension $E$ of $F^*$ by $\Z$, which only depends on the class of $E$ in $\Ext(F^*,\Z)$. The original extended Bloch group is the extended Bloch group associated to the extension of $\C^*$ given by the exponential map. For a large class of fields, including number fields and finite fields, the extended Bloch groups $\widehat\B_E(F)$ are isomorphic, and can be glued together to form an extended Bloch group $\widehat\B(F)$ which only depends on $F$, and admits a natural Galois action.
This is studied in Section~\ref{BhatF}.

By a result of Suslin~\cite{Suslin}, the classical Bloch group $\B(F)$ of an (infinite) field $F$ is isomorphic to the algebraic $K$--group $K_3^{\text{ind}}(F)$ modulo torsion. 
More precisely, Suslin proves that there is an exact sequence
\begin{equation}\label{Suslinexact}
0\to \widetilde{\mu_F}\to K_3^{\textnormal{ind}}(F)\to \B(F)\to 0,
\end{equation}
where $\mu_F$ denotes the roots of unity in $F$, and $\widetilde{\mu_F}$ is the unique non-trivial extension of $\mu_F$ by $\Z/2\Z$ (in characteristic $2$, $\widetilde{\mu_F}=\mu_F$).
Our main result is the following. 

\begin{thm}\label{mainthm}
For every number field $F$, there is a natural isomorphism 
\begin{equation*}\widehat\lambda\colon K_3^{\textnormal{ind}}(F)\cong \widehat\B(F)\end{equation*}
respecting the Galois actions.\qed\end{thm}

In Section~\ref{hypthm} we give the following geometric application generalizing a result of Goncharov~\cite{Goncharov}, who proved the existence of a fundamental class in $K_3^{\ind}(\overline\Q)\otimes\Q$. 
\begin{thm}\label{closedandcusps}
Let $M$ be a complete, oriented, hyperbolic $3$--manifold of finite volume. Let $K$ and $k$ denote the trace field and invariant trace field of $M$. If $M$ is closed, $M$ has a fundamental class $[M]$ in $K_3^{\ind}(K)$ defined up to two-torsion, and satisfying that $2[M]\in K_3^{\ind}(k)$. If $M$ has cusps, there is a fundamental class $[M]$ in $K_3^{\ind}(k)\otimes\Z[\frac{1}{2}]$ such that $8[M]$ is in $K_3^{\ind}(k)$.\qed
\end{thm}

The result is proved using both concrete properties of the extended Bloch group and abstract properties of $K_3^{\ind}(F)$.




There is a regulator map 
\[R\colon K_3^{\textnormal{ind}}(\C)\to \C/4\pi^2\Z.\] 
The regulator is equivariant with respect to complex conjugation, so if $F$ is a number field, we obtain a regulator 
\begin{equation}\label{regonK3}
\widehat B\colon K_3^{\textnormal{ind}}(F)\to (\R/4\pi^2\Z)^{r_1}\oplus (\C/4\pi^2\Z)^{r_2},
\end{equation}
where $r_1$ and $r_2$ are the number of real and (conjugate pairs of) complex embeddings of $F$ in $\C$. This regulator fits into a diagram
\begin{equation*}
\xymatrix{K_3^{\text{ind}}(F)\ar[r]^-{\widehat B}\ar[d]&{(\R/4\pi^2\Z)^{r_1}\oplus (\C/4\pi^2\Z)^{r_2}}\ar[d]\\{\B(F)}\ar[r]^-B& {\R^{r_2}},}
\end{equation*}
where the left vertical map is the map in \eqref{Suslinexact}, and the right vertical map is projection onto the imaginary part.
The lower map $B$ is known as the Borel regulator and has been extensively studied. It is related to hyperbolic volume, and it is known that the image in $\R^{r_2}$ is a lattice whose covolume is proportional to the zeta function of $F$ evaluated at $2$. We refer to Zagier~\cite{Zagier} for a survey. The upper map is much less understood. The real part is related to the Chern-Simons invariant, but little is known about its relations to number theory. 

In section~\ref{BhatF}, we give a concrete formula for $\widehat B$ defined on the extended Bloch group $\widehat\B(F)=K_3^{\textnormal{ind}}(F)$. Elements in $\widehat\B(F)$ are easy to produce, e.g.~using computer software like PARI/GP, and our result can thus provide lots of experimental data for studying the map $\widehat B$. We give an example in Example~\ref{Rexample}.

The torsion in $K_3^{\textnormal{ind}}(F)$ is known to be cyclic of order $w=2\prod p^{\nu_p}$, where
\[\nu_p=\text{max}\{\nu\mid \xi_{p^\nu}+\xi_{p^\nu}^{-1}\in F\}.\]
The product, which is easily seen to be finite, is over all rational primes, and $\xi_{p^\nu}$ is a primitive root of unity of order $p^\nu$. 
This result is due to Merkurjev--Suslin~\cite{MerkurjevSuslin}; see also the survey paper Weibel~\cite{Weibel}. 

In Section~\ref{TorinBhat}, we give explicit elements in $\widehat\B(F)$ generating the torsion. As a corollary, this gives explicit generators of the torsion in the Bloch group. We state this result below. Let $\B(F)_p$ denote the elements in $\B(F)$ of order a power of $p$. By \eqref{Suslinexact}, the order of $\B(F)_p$ is $p^{\nu'_p}$, where $\nu'_p=\nu_p-\text{max}\{\nu\mid \xi_{p^\nu}\in F\}$.

\begin{thm}\label{torsioninB} Let $F$ be a number field and let $p$ be a prime number with $\nu'_p>0$. Let $x$ be a primitive root of unity of order $p^{\nu_p}$. The elements
\begin{gather*}
\begin{aligned}\beta_{p}&=\sum_{k=1}^{p^{\nu_p}}\left[\frac{(x^{k+1}+x^{-k-1})(x^{k-1}+x^{-k+1})}{(x^k+x^{-k})^2}\right],\\ \beta_{2}&=\sum_{k=1}^{2^{\nu_2-1}}\left[\frac{(x^{k+1}-x^{-k})(x^{k-1}-x^{-k+2})}{(x^k-x^{-k+1})^2}\right]
\end{aligned}
\end{gather*}
generate $\B(F)_p$ for odd $p$ and $p=2$, respectively.\qed
\end{thm}

Note that the torsion in the Bloch group comes from a totally real abelian subfield of $F$.

\begin{remark}The torsion in the Bloch group of a number field is related to conformal field theory, and there is an interesting conjecture relating torsion in the Bloch group to modularity of a certain $q$--hypergeometric function. See Nahm~\cite{Nahm}, or Zagier~\cite{Zagier}.\end{remark} 

\begin{ack} I wish to thank Ian Agol, Johan Dupont, Stavros Garoufalidis, Matthias Goerner, Dylan Thurston and, in particular, Walter Neumann for helpful discussions. I also wish to thank Walter Neumann for his comments on earlier drafts of the paper. Parts of this work was done during a visit to the Max Planck Institute of Mathematics, Bonn. I wish to thank MPIM for its hospitality, and for providing an excellent working environment. 
\end{ack}

\section{Preliminaries}\label{Preliminaries}
For an abelian group $A$, we define 
\[\wedge^2(A)=A\otimes_\Z A\big/ \langle a\otimes b+b\otimes a\rangle.\] 
Note that $2a\wedge a=0$, but $a\wedge a$ is generally not $0$.

For a set $X$, we let $\Z[X]$ denote the free abelian group generated by $X$.
\subsection{The classical Bloch group}
Let $F$ be a field and let $F^*$ be the multiplicative group of units in $F$.
Consider the set of \emph{five term relations}
\[\FT=\big\{\big(x,y,\frac{y}{x},\frac{1-x^{-1}}{1-y^{-1}},\frac{1-x}{1-y}\big)\bigm\vert
x\neq y \in F\setminus\{0,1\}\big\}.\]
One can show that there is a chain complex
\begin{equation}\label{BlochSuslin}
\xymatrix{{\Z[\FT]}\ar[r]^-\rho&{\Z[F\setminus \{0,1\}]}\ar[r]^-\nu
&{\wedge^2(F^*)}},
\end{equation}
with maps defined by
\begin{gather*}
\rho([z_0,\ldots,z_4])=[z_0]-[z_1]+[z_2]-[z_3]+[z_4],\\
\nu([z])=z\wedge (1-z).
\end{gather*}
\begin{remark} By Matsumoto's theorem, the cokernel of $\nu$ is $K_2(F)$.\end{remark}

\begin{defn}
The \emph{Bloch group} of $F$ is the quotient $\B(F)=\Ker(\nu)/\Image(\rho)$. It is a subgroup of the \emph{pre-Bloch group} 
 $\Pre(F)=\Z[F\setminus\{0,1\}]/\Image(\rho)$. 
\end{defn}
\subsection{The extended Bloch group of $\C$}
The original reference is Neumann~\cite{Neumann}; see also Dupont--Zickert \cite{DupontZickert} and
Goette--Zickert \cite{GZ}. We stress that our extended Bloch group is what Neumann calls the
\emph{more} extended Bloch group \cite[Section 8]{Neumann}.

Consider the set 
\[\widehat \C=\big\{(w_0,w_1)\in \C^2\bigm \vert \exp(w_0)+\exp(w_1)=1\big\}.\]
We will refer to elements of $\widehat \C$ as \emph{flattenings}. 
We can view $\widehat \C$ as the Riemann surface for the multivalued function $(\log(z),\log(1-z))$, 
and we can thus write a flattening as $[z;2p,2q]=(\log(z)+2p\pi i,\log(1-z)+2q\pi i)$. This notation depends on a choice of logarithm branch which we will
fix once and for all. The map $\pi\colon \widehat \C\to \C\setminus\{0,1\}$ taking a
flattening $[z;2p,2q]$ to $z$ is the universal
abelian cover of $\C\setminus\{0,1\}$. 

\begin{remark}\label{ChangeSign} Neumann considered the Riemann surface of $(\log(z),-\log(1-z))$,
and considered a flattening $[z;2p,2q]$ as a triple $(w_0,w_1,w_2)$, with $w_0=\log(z)+2p\pi
i$, $w_1=-\log(1-z)+2q\pi i$ and $w_2=-w_1-w_0$. One translates between the two
definitions by changing the sign of $w_1$, or equivalently, by changing the sign of
$q$.
\end{remark}

Let $\FT_0=\big\{(x_0,\ldots,x_4)\in \FT\bigm \vert 0<x_1<x_0<1\big\}$, and define
the set of \emph{lifted five term relations} $\widehat{\FT}\subset (\widehat \C)^5$ to be the component of 
the preimage of $\FT$ containing all points $\big([x_0;0,0],\dots,[x_4;0,0]\big)$ with $(x_0,\dots, x_4)\in \FT_0$.

There is a chain complex
\begin{equation}
\xymatrix{{\Z[\widehat{\FT}]}\ar[r]^-{\widehat\rho}&{\Z[\widehat \C]}\ar[r]^-{\widehat\nu}
&{\wedge^2(\C),}}
\end{equation}
with maps defined by
\begin{gather*}
\widehat\rho([(w_0^0,w_1^0),\ldots,(w_0^4,w_1^4)])=\sum_{i=0}^4(-1)^i[(w_0^i,w_1^i)],\\
\widehat\nu([(w_0,w_1)])=w_0\wedge w_1.
\end{gather*}
\begin{defn} The \emph{extended Bloch group} is the quotient $\widehat \B(\C)=\Ker(\widehat\nu)/\Image(\widehat
\rho)$. It is a subgroup of the \emph{extended pre-Bloch group} $\widehat\Pre(\C)=\Z[\widehat\C]/\Image(\widehat \rho)$. 
\end{defn}


 \begin{thm}\label{bigdiag} Let $\mu_{\C}$ denote the roots of unity in $\C^*$. There is a commutative diagram as below with exact rows and columns.\qed\end{thm}
\begin{equation}
\cxymatrix{{&0\ar[d]&0\ar[d]&0\ar[d]&&\\0\ar[r]&\mu_{\C}\ar[r]\ar[d]^-{\chi}&{\C^*}\ar[r]\ar[d]^{\chi}&{\C^*/\mu_{\C}}\ar[r]\ar[d]^\beta&0\ar[d]&\\
0\ar[r]&{\widehat \B(\C)}\ar[r]\ar[d]^-\pi&{\widehat \Pre(\C)}\ar[r]^-{\widehat \nu}\ar[d]^-\pi&{\wedge^2(\C)}\ar[r]\ar[d]^\epsilon&
{K_2(\C)}\ar@2{-}[d]\ar[r]&0\\0\ar[r]&{\B(\C)}\ar[r]\ar[d]&{\Pre(\C)}\ar[r]^-\nu\ar[d]&{\wedge^2(\C^*)\ar[d]}\ar[r]&{K_2(\C)}\ar[r]\ar[d]&0\\&0&0&0&0}}
\end{equation}

We refer to Goette--Zickert \cite{GZ} or Section~\ref{BhatE} below for the definition of $\chi$. The
other maps are defined as follows:
\begin{align*}
\beta([z])&=\log(z)\wedge 2\pi i;\\
\epsilon(w_1\wedge w_2)&=\exp(w_1)\wedge \exp(w_2);\\
\pi([z;2p,2q])&=[z].
\end{align*}

\begin{remark}
There is a similar definition of $\widehat\Pre(F)$ and $\widehat\B(F)$ for a subfield $F$ of $\C$. Theorem~\ref{bigdiag} holds if one replaces $\wedge^2(\C)$ by $\wedge^2(\{w\in
\C\mid \exp(w)\in F^*\})$ and $\mu_\C$ with $\widetilde{\mu_F}=\{w\in\C\mid w^2\in \mu_F\}$. We will generalize to arbitrary fields below.
\end{remark}
\subsubsection{The regulator}
The function $R\colon \widehat \C\to \C/4\pi^2$, given by
\begin{equation}\label{regulator}
[z;2p,2q]\mapsto \Li_2(z)+\frac{1}{2}(\log(z)+2p\pi i)(\log(1-z)-2q\pi i)-\pi^2/6
\end{equation}
is well defined and holomorphic, see e.g.~Neumann~\cite{Neumann} or Goette--Zickert~\cite{GZ}.

It is well known that $L(z)=\Li_2(z)+\frac{1}{2}\log(z)\log(1-z)-\pi^2/6$ satisfies 
\[\sum_{i=1}^4(-1)^i L(z_i)=0 \text{ for } (z_0,\dots,z_4)\in \FT_0,\] and it thus follows by analytical
continuation that $R$ gives rise to a function 
\begin{equation}R\colon \widehat \Pre(\C)\to \C/4\pi^2\Z.
\end{equation} 

We briefly describe a more elegant definition of $R$ due to Don Zagier \cite{Zagier}:
The derivative of $\Li_2(z)$ is $-\log(1-z)/z$.  
It follows that the function $F(x)=\Li_2(1-e^x)$ has derivative $F'(x)=xe^x/(1-e^x)$.
Since this function is meromorphic with simple poles at $2\pi i n$, $n\in \Z$, with corresponding residues $-2\pi i n$, it follows that
$F$ defines a single valued function on $\C\setminus\{2\pi i\Z\}$ with values in $\C/4\pi^2\Z$. 
We can now define 
\begin{equation}\label{Zagierregulator}R\colon \widehat \C\to \C/4\pi^2,\qquad (w_0,w_1)\mapsto F(w_1)+\frac{w_0w_1}{2}-\pi^2/6.\end{equation}
We leave it to the reader to show that this definition of $R$ agrees with the one above.

\subsection{Algebraic $K$--theory and homology of linear groups}\label{AlgKreview}
We give a brief review of the results that we shall need. 

Let $F$ be a field. The algebraic $K$--groups are defined by $K_i(F)=\pi_i(B\GL(F)^+)$. The Milnor $K$--groups $K_*^M(F)$ are defined as the quotient of the tensor algebra of $F^*$ by the two-sided ideal generated by $a\otimes(1-a)$. There is a natural map 
\[K_i^M(F)\to K_i(F)\] 
whose cokernel, by definition, is the \emph{indecomposable} $K$--group $K_i^{\text{ind}}(F)$.

For $F=\C$, there is a \emph{regulator}~$R$ defined as the composition
\[\xymatrix{{K_{2k-1}(\C)}\ar[r]^-H&{H_{2k-1}(\GL(\C))}\ar[r]^-{\hat c_k}&{\C/(2\pi i)^k\Z}},\]
where $H$ is the Hurewicz map, and $\hat c_k$ is the universal Cheeger--Chern--Simons class. 
It is well known that $R$ is $0$ on the image of $K_{2k-1}^M(\C)$, so $R$ induces a regulator 
\[K_{2k-1}^{\text{ind}}(\C)\to \C/(2\pi i)^k\Z.\] 

\begin{thm}[Suslin~\cite{Suslin1046}]\label{Suslinstability}
For any field $F$, there is an isomorphism
\[H_n(\GL(n,F))\cong H_n(\GL(F))\]
induced by inclusion.\qed
\end{thm}

\begin{thm}[Sah~\cite{Sah}]\label{Sahthm}
$K_3^{\ind}(\C)$ is a direct summand of $K_3(\C)$ and the Hurewicz map $H$ induces an isomorphism $K_3^{\ind}(\C)\cong H_3(\SL(2,\C))$.\qed
\end{thm}

\begin{thm}[Goette--Zickert~\cite{GZ}; see also Neumann~\cite{Neumann}]\label{H3eqBhat}
There is a canonical isomorphism $H_3(\SL(2,\C))\cong\widehat\B(\C)$. Under this isomorphism, $\hat c_2$ corresponds to the map $R$ in \eqref{regulator}.\qed
\end{thm}

\begin{thm}[Dupont--Sah~\cite{DupontSah}]\label{torsiondiag}
The diagonal map $x\mapsto \bigl(\begin{smallmatrix}x&0\\0&x^{-1}\end{smallmatrix}\bigr)$ induces
an injection $H_3(\mu_\C)\to H_3(\SL(2,\C))$ onto the torsion subgroup of $H_3(\SL(2,\C))$.\qed
\end{thm}

\section{The extended Bloch group of an extension}\label{BhatE}
Let $F$ be a field and let 
\begin{equation*}
\xymatrix{E\colon 0\ar[r]&{\Z}\ar[r]^-\iota &E\ar[r]^-\pi &F^*\ar[r]& 0}
\end{equation*} be an extension of $F^*$ by $\Z$. We stress that the letter $E$ is used both to denote the extension and the middle group. As we shall see, most of the results in Neumann~\cite{Neumann} and Goette--Zickert \cite{GZ} can be formulated in this purely algebraic setup. 


\begin{defn} The set of \emph{(algebraic) flattenings} is the set
\[\widehat F_E=\big\{(e,f)\in E\times E\bigm\vert \pi(e)+\pi(f)=1\in F\big\}.\]
The map $(e,f)\mapsto \pi(e)$ induces a surjection $\pi\colon \widehat{F}_E\to
F\setminus\{0,1\}$, and we say that $(e,f)$ is a \emph{flattening} of $\pi(e)$.
\end{defn}
Recall the set of five term relations \[\FT=\big\{\big(x,y,\frac{y}{x},\frac{1-x^{-1}}{1-y^{-1}},\frac{1-x}{1-y}\big)\bigm\vert
x\neq y \in F\setminus\{0,1\}\big\}.\]
\begin{defn}\label{fiveeqdef}
The set of \emph{lifted five term relations} $\widehat{\FT}_E\subset (\widehat
F_E)^5$ is the set of tuples of flattenings 
$\big((e_0,f_0),\dots,(e_4,f_4)\big)$ satisfying 
\begin{gather}\label{fiveeq}
\begin{aligned}
e_2&=e_1-e_0\\e_3&=e_1-e_0-f_1+f_0\\f_3&=f_2-f_1\\
e_4&=f_0-f_1\\f_4&=f_2-f_1+e_0.
\end{aligned}
\end{gather}
\end{defn}
If $\big((e_0,f_0),\dots,(e_4,f_4)\big)\in\widehat{\FT}_E$, where $(e_i,f_i)$ is a flattening of $x_i\in F\setminus\{0,1\}$, then \eqref{fiveeq} implies that 
\[x_2=\frac{x_1}{x_0},\quad x_3=\frac{x_1}{x_0}\frac{(1-x_0)}{(1-x_1)}=\frac{1-x_0^{-1}}{1-x_1^{-1}},\quad x_4=\frac{1-x_0}{1-x_1}.\]
Hence, a lifted five term relation is indeed a lift of a five term relation. On the other hand, if $(x_0,\dots,x_4)\in\FT$ it is not difficult to check that there exist flattenings $(e_i,f_i)$ satisfying~\eqref{fiveeq}.
Hence, the map $\pi\colon\widehat{\FT}_E\to \FT$ is surjective. 

Consider the complex
\begin{equation}\label{ExtBlochSuslin}
\xymatrix{{\Z[\widehat{\FT}_E]}\ar[r]^-{\widehat\rho}&{\Z[\widehat
F_E]}\ar[r]^-{\widehat\nu}&\wedge^2(E)},
\end{equation}
with maps defined by 
\begin{gather*}
\widehat\rho\big((e_0,f_0),\dots,(e_4,f_4)\big)=\sum_{i=0}^4(-1)^i(e_i,f_i),\\
\widehat\nu(e,f)=e\wedge f.
\end{gather*}

\begin{lemma}The complex \eqref{ExtBlochSuslin} is a chain complex, 
i.e.~$\widehat{\nu}\circ\widehat{\rho}=0$. 
\end{lemma}
\begin{proof}
Let $\alpha=\big((e_0,f_0),\dots,(e_4,f_4)\big)\in \widehat{\FT}_E$.
Using \eqref{fiveeq} we have
\begin{align*}
\widehat\nu\circ \widehat\rho(\alpha)=&
\sum_{i=0}^4(-1)^ie_i\wedge f_i=
e_0\wedge f_0-e_1\wedge f_1+(e_1-e_0)\wedge
f_2\\&\quad-(e_1-e_0-f_1+f_0)\wedge (f_2-f_1)+(f_0-f_1)\wedge
(f_2-f_1+e_0)\\
=&e_0\wedge(f_0-f_1)+(f_0-f_1)\wedge e_0=0\in\wedge^2(E).\qedhere
\end{align*}
\end{proof}

\begin{defn} 
The \emph{extended pre-Bloch
group} $\widehat\Pre_E(F)$ is the quotient $\Z[\widehat F_E]/\Image(\widehat\rho)$. The \emph{extended Bloch group} $\widehat\B_E(F)$ is the quotient $\Ker(\widehat \nu)/\Image(\widehat\rho)$. 
\end{defn}

\begin{ex}\label{extension}
The extended Bloch group $\widehat\B(\C)$ is the extended Bloch group associated to the extension
\begin{equation}
\xymatrix{0\ar[r]&{\Z}\ar[r]^-{2\pi i}&{\C}\ar[r]^-{\exp}&{\C}^*\ar[r]&0.}
\end{equation}
\end{ex}

The extended groups fit together with the classical groups in a diagram
\[\xymatrix{{\widehat\B_E(F)}\ar[r]\ar[d]&{\widehat\Pre_E(F)}\ar[d]\ar[r]^-{\widehat\nu}&\wedge^2(E)\ar[d]\\{\B(F)}\ar[r]&{\Pre(F)}\ar[r]^-\nu&\wedge^2(F^*),}\]
where the vertical maps are surjections. 

\subsection{Relations in the extended Bloch group}
We now derive some relations in $\widehat{\Pre}_E(F)$. We encourage the reader to compare with Neumann~\cite{Neumann} and
Goette--Zickert~\cite{GZ} where similar relations are derived in $\widehat{\Pre}(\C)$ using analytic
continuation.

In the following we will regard $\Z$ as a subgroup of $E$. Consider the set 
\begin{multline} V=\big\{\big((p_0,q_0),(p_1,q_1),(p_1-p_0,q_2),(p_1-p_0-q_1+q_0,q_2-q_1),\\
(q_0-q_1,q_2-q_1+p_0)\big)\bigm \vert p_0,p_1,q_0,q_1,q_2\in \Z\big\}\subset (\Z\times \Z)^5,\end{multline}
also considered by Neumann~\cite[Definition 2.2]{Neumann}.

By \eqref{fiveeq} it follows that componentwise addition gives rise to an action
\[+\colon\widehat{\FT}_E\times V\to \widehat{\FT}_E, \quad (\alpha,v)\mapsto \alpha + v.\]
 

\begin{lemma}\label{GZ3.3} Let $q,q',\bar q,\bar q'$ be integers satisfying $q-q'=\bar q-\bar q'$.
For each flattening $(e,f)\in \widehat F_E$ we have
\begin{equation}
(e,f+q)-(e,f+q')=(e,f+\bar q)-(e,f+\bar q')\in \widehat{\Pre}_E(F).
\end{equation} 
\end{lemma}
\begin{proof}
Let $\alpha=\big((e_0,f_0),\dots,(e_4,f_4)\big)\in \widehat{\FT}_E$. For each integer $r$, consider the element $v_r\in V$ given by
\[v_r=\big((0,r),(0,r),(0,r),(0,0),(0,0)\big).\]
The relation $\widehat\rho(\alpha+v_q)-\widehat\rho(\alpha+v_{q'})=0\in \widehat\Pre_E(F)$
can be written as
\begin{multline}\label{V0eq}(e_0,f_0+q)-(e_0,f_0+q')-\big((e_1,f_1+q)-(e_1,f_1+q')\big)\\
+(e_2,f_2+q)-(e_2,f_2+q')=0\in \widehat{\Pre}_E(F).\end{multline}Let $\beta=\alpha+\big((0,0),(0,s),(0,s),(-s,0),(-s,0)\big)$, where $s=q-\bar q=q'-\bar q'$. Then $\beta\in \widehat{\FT}_E$, and the relation $\widehat\rho(\beta+v_{\bar q})-\widehat\rho(\beta+v_{\bar q'})=0\in \widehat \Pre_E(F)$ becomes
\begin{multline}\label{V0eq2}(e_0,f_0+\bar q)-(e_0,f_0+\bar q')-\big((e_1,f_1+q)-(e_1,f_1+q')\big)\\
+(e_2,f_2+q)-(e_2,f_2+q')=0\in \widehat{\Pre}_E(F).\end{multline}
The result now follows by subtracting \eqref{V0eq2} from \eqref{V0eq}.
\end{proof}

\begin{cor}\label{chidefn} Let $e\in E\setminus \Z$. The element
\[(e,f+1)-(e,f)\in\widehat\Pre_E(F)\]
is independent of $f$ whenever $(e,f)$ is in $\widehat F_E$.\qed
\end{cor}

Using Corollary~\ref{chidefn}, we can define a map 
\begin{equation}\label{chieq}\chi\colon E\setminus\Z\to \widehat\Pre_E(F),\quad e\mapsto (e,f+1)-(e,f).\end{equation} 
\begin{lemma}\label{chihomo}
Suppose $e,e'$ and $e+e'$ are elements in $E\setminus\Z$. We have
\begin{equation}
\chi(e)+\chi(e')=\chi(e+e').
\end{equation}
\end{lemma}
\begin{proof}
This follows from \eqref{V0eq} after noting that $e_1=e_0+e_2$.
\end{proof}

The following is elementary.
\begin{lemma}\label{exttohomo}
Let $G$ and $G'$ be groups and let $H$ be a subgroup of $G$ of index greater than $2$.
Suppose $\phi\colon G\setminus H\to G'$ is a map satisfying $\phi(g_1g_2)=\phi(g_1)\phi(g_2)$ whenever both sides are defined. Then $\phi$ extends uniquely to a homomorphism $\phi\colon G\to G'$. \qed
\end{lemma}


\begin{cor}\label{extofchi}
The map $\chi\colon E\setminus\Z\to \widehat\Pre_E(F)$ extends to a homomorphism defined on all of $E$.\qed
\end{cor}

\begin{lemma}\label{N7.3}
For any two flattenings $(e,f),(g,h)\in \widehat F_E$ we have
\[(e,f)+(f,e)=(g,h)+(h,g)\in \widehat \B_E(F).\]
\end{lemma}
\begin{proof}
It follows from \eqref{fiveeq} that 
$\big((e_0,f_0),(e_1,f_1),(e_2,f_2),(e_3,f_3),(e_4,f_4)\big)\in \widehat{\FT}_E$
if and only if $\big((f_1,e_1),(f_0,e_0),(e_4,f_4),(e_3,f_3),(e_2,f_2)\big)\in \widehat{\FT}_E.$
Subtracting the two relations in $\widehat{\Pre}_E(F)$ yields
\[(e_0,f_0)-(e_1,f_1)=(f_1,e_1)-(f_0,e_0)\in \widehat{\Pre}_E(F),\]
from which the claim follows. Since $e\wedge f+f\wedge e=0\in \wedge^2(E)$, the element lies in $\widehat\B_E(F)$.
\end{proof}

\begin{lemma}\label{kappadef} 
The homomorphism $\chi\colon E\to \widehat\Pre_E(F)$ takes $2\Z\subset E$ to $0$.
\end{lemma}

\begin{proof}
Let $e\in E$ be any element not in $\Z$. 
The result follows from the computation
\begin{gather}\label{2chi0}
\begin{split}\chi(1)&=\chi(e+1)-\chi(e)\\&=(e+1,f+1)-(e+1,f)-(e,f+1)+(e,f)\\&=
-(f+1,e+1)+(f,e+1)+(f+1,e)-(f,e)\\&=-\chi(f+1)+\chi(f)\\&=-\chi(1),
\end{split}
\end{gather}
where the third equality follows from Lemma~\ref{N7.3}.
\end{proof}


\begin{thm}\label{preblochrel}
There is an exact sequence
\begin{equation}\label{kernelCmod4}
\xymatrix{{E/2\Z}\ar[r]^-\chi&{\widehat\Pre_E(F)}\ar[r]^-\pi&{\Pre(F)}\ar[r]&0.}
\end{equation}
\end{thm}
\begin{proof}
It is clear that $\pi\circ \chi=0$ and that $\pi$ is surjective. Since every five term relation lifts to a lifted five term relation, the kernel of $\pi$ must be generated by elements of the form $(e+p,f+q)-(e,f)$.
By Lemma~\ref{GZ3.3} we have,
\begin{multline}\label{chi1}(e,f+q)=(e,f+q-1)+(e,f+1)-(e,f)\\=q(e,f+1)-q(e,f)+(e,f)=q\chi(e)+(e,f),\end{multline}
and we see that $(e,f+q)-(e,f)$ is in $\Image(\chi)$.
Using Lemma~\ref{N7.3} we similarly obtain
\begin{equation}\label{chi2}(e+p,f)-(e,f)=-p\chi(f)\in\Image(\chi),\end{equation}
and the result follows.\end{proof}

\begin{remark}
We do not know if $\chi$ is injective, but we expect this to be the case. In the next section, we show that $\chi$ is injective if $F$ is a number field and $E$ is a primitive extension.
\end{remark}

\begin{remark}\label{CstarandEmod2}
If $E$ is the extension in Example~\ref{extension}, the exact sequence~\eqref{kernelCmod4} is equivalent to the corresponding exact sequence in Theorem~\ref{bigdiag} using the identification of $\C^*$ with $E/2\Z=\C/4\pi i\Z$ taking $z\in \C^*$ to $-2\log(z)\in \C/4\pi i\Z$. The restriction of the regulator~\eqref{regulator} to $\C/4\pi i\Z\subset \widehat \Pre(\C)$ is multiplication by~$-\pi i$. 
\end{remark}

\begin{lemma}\label{Rformula} The following equality holds in $\widehat\Pre_E(F)$.
\begin{equation*}(e+p,f+q)-(e,f)=\chi(qe-pf+pq).\end{equation*}
\end{lemma}
\begin{proof}
This is an easy consequence of \eqref{chi1} and \eqref{chi2}.
\end{proof}

\subsection{Functoriality}
Let $F_1$ and $F_2$ be fields and let $E_1$ and $E_2$ be extensions of $F_1^*$ and $F_2^*$ by $\Z$. 
\begin{defn}A map $\Psi\colon E_1\to E_2$ of extensions is called a \emph{covering} if the base homomorphism $\Psi\colon F_1^*\to F_2^*$ extends to an embedding of $F_1$ in $F_2$. Two coverings are \emph{equivalent} if they cover the same embedding.
\end{defn}
A covering $\Psi\colon E_1\to E_2$
gives rise to a chain map
\begin{equation*}
\xymatrix{{\Z[\widehat{\FT}_{E_1}]}\ar[r]^-{\widehat \rho_1}\ar[d]^{\Psi_*}&{\Z[\widehat F_{E_1}]}\ar[r]^-{\widehat\nu_1}\ar[d]^{\Psi_*}&{\wedge^2(E_1)}\ar[d]^{\Psi\wedge\Psi}\\
{\Z[\widehat{\FT}_{E_2}]}\ar[r]^-{\widehat \rho_2}&{\Z[\widehat F_{E_2}]}\ar[r]^-{\widehat\nu_2}&{\wedge^2(E_2)}}
\end{equation*}
defined by taking an algebraic flattening $(e,f)$ to $(\Psi(e),\Psi(f))$.
We thus obtain maps
\begin{equation}
\Psi_*\colon \widehat\Pre_{E_1}(F_1)\to \widehat \Pre_{E_2}(F_2),\quad \Psi_*\colon \widehat\B_{E_1}(F_1)\to \widehat \B_{E_2}(F_2),
\end{equation}
satisfying the usual functoriality properties.

\begin{lemma}\label{ktoksq} Let $\Psi\colon E_1\to E_2$ be a covering. There is a commutative diagram of exact sequences 
\[\xymatrix{{E_1/2\Z}\ar[r]^-\chi\ar[d]^{\Psi_*}&{\widehat\Pre_{E_1}(F_1)}\ar[r]^-\pi\ar[d]^{\Psi_*}&{\Pre(F_1)}\ar[r]\ar[d]^{\Psi_*}&0\\{E_2/2\Z}\ar[r]^-\chi&{\widehat\Pre_{E_2}(F_2)}\ar[r]^-\pi&{\Pre(F_2)}\ar[r]&0.}\]
The map $\Psi_*\colon E_1/2\Z\to E_2/2\Z$ takes $e\in E_1/2\Z$ to $\Psi(1)\Psi(e)\in E_2/2\Z$, where $\Psi(1)\Psi(e)$ is defined using the natural action of $\Z\subset E_2$ on $E_2$ by multiplication.
\end{lemma}
\begin{proof}
Exactness follows from Theorem~\ref{preblochrel}. Commutativity of the right square is obvious, and commutativity of the left square follows from the computation
\[\Psi_*(\chi(e))=\big(\Psi(e),\Psi(f)+\Psi(1)\big)-\big(\Psi(e),\Psi(f)\big)=\chi\big(\Psi(1)\Psi(e)\big),\]
where the second equality follows from Lemma~\ref{Rformula}.
\end{proof}
 
We wish to prove that the induced map $\Psi_*\colon\widehat\B_{E_1}(F_1)\to\widehat\B_{E_2}(F_2)$ of a covering only depends on the underlying embedding. The result below is elementary.


\begin{lemma}\label{TorELemma}
Let $F$ be a field. The torsion subgroup $\Tor(F^*)$ is isomorphic to a subgroup of $\Q/\Z$. 
For any extension $E$ of $F^*$ by $\Z$, the same is true for $\Tor(E)$.\qed
\end{lemma}

\begin{lemma}\label{zeroinwedge} An element $\sum_i n_ie_i\wedge f_i$ is zero in $\wedge^2(E)$ if and only if we can write
\begin{equation}\label{eiandfi}e_i=k_iw+\sum\nolimits_j r_{ij}p_j,\quad f_i=l_iw+\sum\nolimits_j s_{ij}p_j,
\end{equation} 
where $p_j\in E$, $w\in E$ is a torsion element, and the integers $s_{ij}$,$r_{ij}$,$k_i$ and $l_i$ satisfy
\begin{enumerate}[(i)]
\item\label{wedge1} $\sum_i n_ir_{ij}s_{ij}$ is even for each $j$;
\item\label{wedge2} $\sum_i n_i(r_{ij}s_{ik}-r_{ik}s_{ij})=0$ for each $j\neq k$;
\item\label{wedge3} $\sum_i n_i(l_ir_{ij}-k_is_{ij})$ is divisible by $\ord(w)$ for each $j$;
\item\label{wedge4} $\sum_i n_ik_il_i$ is even.
\end{enumerate}
\end{lemma}
\begin{proof}

Let $\alpha=\sum_i n_ie_i\wedge f_i$. Since $\wedge^2$ commutes with direct limits, there exists a finitely generated subgroup $H$, containing the $e_i$'s and $f_i$'s, such that $\alpha$ is zero in $\wedge^2(E)$ if and only if $\alpha$ is zero in $\wedge^2(H)$. Let $p_j$ be free generators of $H$ and let $w$ be a generator of the torsion (which is cyclic by Lemma~\ref{TorELemma}). Write the $e_i$'s and $f_i$'s as in \eqref{eiandfi}. When expanding $\alpha\in\wedge^2(H)$, the coefficients of $p_j\wedge p_j$, $p_j\wedge p_k$, $p_j\wedge w$ and $w\wedge w$ of $\wedge^2(H)$ are given, respectively, by \eqref{wedge1}-\eqref{wedge4}. Hence, $\alpha=0\in\wedge^2(H)$ if \eqref{wedge1}-\eqref{wedge4} hold.
On the other hand, if $\alpha=0\in\wedge^2(H)$, \eqref{wedge1}-\eqref{wedge3} hold, whereas \eqref{wedge4} may fail if $w \wedge w=0\in \wedge^2(H)$. This happens if and only if $w$ is $2$--divisible, in which case, we may replace $w$ by a half if necessary, to make \eqref{wedge4} hold as well.
\end{proof}

\begin{thm}\label{isoonBhat} If $\Psi_1,\Psi_2\colon E_1\to E_2$ are equivalent coverings then
\begin{equation}
\Psi_{1*}=\Psi_{2*}\colon \widehat\B_{E_1}(F_1)\to \widehat \B_{E_2}(F_2).
\end{equation}
\end{thm}

\begin{proof}For notational simplicity, we assume that $\Psi_1=\id$, and drop the subscripts of $F_i$ and $E_i$. This case is sufficient for most of the applications. We leave the general proof, which differs only in notation, to the reader. 

Let $\phi=\Psi_2-\Psi_1\colon E_1\to\Z\subset E_2$ and let $\alpha=\sum n_i(e_i,f_i)\in\widehat\B_{E}(F)$. We wish to prove that 
\begin{equation}
\Delta:=\Psi_{2*}(\alpha)-\Psi_{1*}(\alpha)=\sum\nolimits_i n_i\big(\big(e_i+\phi(e_i),f_i+\phi(f_i)\big)-(e_i,f_i)\big) 
\end{equation}
is zero in $\widehat \B_{E}(F)$. By Lemma~\ref{Rformula}, we have
\begin{equation}\label{RDeltaone}
\Delta=\chi\Big(\sum\nolimits_i n_i\big(\phi(f_i)e_i-\phi(e_i)f_i+\phi(e_i)\phi(f_i)\big)\Big).
\end{equation}
Since $\alpha\in\widehat \B_{E}(F)$, we have $\widehat \nu(\alpha)=\sum n_ie_i\wedge f_i=0\in \wedge^2(E)$.
By Lemma~\ref{zeroinwedge}, we can write 
\[e_i=k_iw+\sum\nolimits_j r_{ij}p_j,\quad f_i=l_iw+\sum\nolimits_j s_{ij}p_j,\]
where \eqref{wedge1}-\eqref{wedge4} in Lemma~\ref{zeroinwedge} are satisfied. Let $a_j=\phi(p_j)$. Let $T$ denote the sum in \eqref{RDeltaone}, i.e.~$\chi(T)=\Delta$. Since $w$ is a torsion element, $\phi(w)=0$, so by \eqref{RDeltaone} we have
\begin{multline}
T=\sum\nolimits_i n_i\Big(\sum\nolimits_js_{ij}a_j\big(\sum\nolimits_jr_{ij}p_j+k_iw\big)-\\
\sum\nolimits_jr_{ij}a_j\big(\sum\nolimits_js_{ij}p_j+l_iw\big)+\sum\nolimits_j r_{ij}a_j\sum\nolimits_j s_{ij}a_j\Big).
\end{multline}
When expanding the sum, the coefficient of $p_k$ is
\[\sum\nolimits_in_i\big(\sum\nolimits_j s_{ij}a_j\big)r_{ik}-\sum\nolimits_in_i\big(\sum\nolimits_j r_{ij}a_j\big)s_{ik}=-\sum\nolimits_j a_j\sum\nolimits_i n_i(r_{ij}s_{ik}-s_{ij}r_{ik}),\]
which is zero by Lemma~\ref{zeroinwedge}, \eqref{wedge2}.
Similarly, the coefficient of $w$ is
\[\sum\nolimits_i n_i\big(\sum\nolimits_j a_jk_is_{ij}-\sum\nolimits_j a_jl_ir_{ij}\big)=-\sum\nolimits_ja_j\sum\nolimits_in_i(l_ir_{ij}-k_is_{ij}),\]
which, by Lemma~\ref{zeroinwedge}, \eqref{wedge3}, is divisible by the order of $w$.
Finally, the remaining terms sum to the integer
\[\sum\nolimits_in_i\Big(\sum\nolimits_jr_{ij}a_j\sum\nolimits_js_{ij}a_j\Big)=\sum\nolimits_in_i\Big(\sum_{j\neq k}r_{ij}s_{ik}a_ja_k\Big)+\sum\nolimits_ja_j^2\sum\nolimits_in_ir_{ij}s_{ij},\]
which is even by Lemma~\ref{zeroinwedge}, \eqref{wedge1} and \eqref{wedge2}. 
Hence, $T$ is zero in $E/2\Z$, so $\Delta=\chi(T)$ is also zero.
\end{proof}

\begin{cor} Up to canonical isomorphism, the extended Bloch group $\widehat \B_E(F)$ depends only on the class of $E$ in $\Ext(F^*,\Z)$.
\end{cor}
\begin{proof}
If $E_1=E_2\in\Ext(F^*,\Z)$, there must exist a covering $\Psi\colon E_1\to E_2$ of the identity on $F$. Since any two such coverings are equivalent, the result follows.
\end{proof}

\subsection{General properties of extensions}
Let $\mu_F\subset F^*$ denote the roots of unity in $F$.
For a prime number $p$ let $\mu_p$ denote the $p$th roots of unity in $\mu_F$, and let $\mu_{p^\infty}$ be the subgroup of roots of unity of order a power of $p$. 
Note that $\mu_F=\oplus \mu_{p^\infty}$. Also note that up to isomorphism, $\mu_{p^\infty}$ is either $\Z/p^n\Z$ or $\Z[1/p]/\Z$.
We have the following classification of $\Z$--extensions of $\Z/p^n\Z$ and $\Z[1/p]/\Z$. The proofs are elementary and left to the reader.
\begin{lemma}\label{ZmodpExt} We have $\Ext(\Z/p^n\Z,\Z)=\Z/p^n\Z$. Let $0\leq k \leq n-1$ and let $x$ be an integer which is not divisible by $p$. 
The non-trivial extensions are given explicitly by 
\begin{equation}\label{extensionZmodp}
\xymatrix{0\ar[r]&{\Z}\ar[r]^-\iota&{\Z\oplus\Z/p^k\Z}\ar[r]^-\pi&{\Z/p^n\Z}\ar[r]&0},
\end{equation}
where the maps are defined by 
\begin{equation}\iota(1)=(p^{n-k},-x),\quad \pi(a,b)=xa+p^{n-k}b.\end{equation} 
The equivalence class of \eqref{extensionZmodp} only depends on the value of $x$ in $(\Z/p^{n-k}\Z)^*$, 
and the order of the extension in $\Ext(\Z/p^n\Z,\Z)$ is $p^{n-k}$.\qed
\end{lemma}

\begin{lemma}\label{extensionpadic}
We have $\Ext(\Z[\frac{1}{p}]/\Z,\Z)=\Z_p$, the $p$--adic integers. Let $0\leq k$ be an integer and let $y$ be a $p$--adic integer with discrete valuation $k$.
The non-trivial extensions are given explicitly by
\begin{equation}\label{padicext}
\xymatrix{0\ar[r]&{\Z}\ar[r]^-\iota&{\Z[\frac{1}{p}]\oplus\Z/p^k\Z}\ar[r]^-{\pi}&{\Z[\frac{1}{p}]\big/\Z}\ar[r]&0,}
\end{equation}
where the maps are defined by 
\begin{equation}\iota(1)=(1/p^k,-y/p^k),\quad \pi(a,b)=\frac{ay+b}{p^k}.\end{equation}\qed
\end{lemma}

\begin{defn} An extension $E$ is called \emph{primitive} if $E$ is torsion free. If $\mu_p$ is non-trivial, and if the restriction $E_{\mu_{p^\infty}}$ is torsion free, we say that $E$ is \emph{$p$--primitive}. 
\end{defn}

We state some elementary corollaries of Lemma~\ref{ZmodpExt} and Lemma~\ref{extensionpadic}.

\begin{cor}\label{Exteq}
$E$ is $p$--primitive if and only if $1\in E$ is divisible by $p$. If so, $1$ is divisible by $\vert \mu_{p^\infty}\vert$ (if $\vert \mu_{p^\infty}\vert=\infty$, $1$ is divisible by $p$ infinitely often).\qed
\end{cor}

\begin{cor}\label{cond2cor} Suppose $\mu_{F}$ is finite. Then $E$ is primitive if and only if $E_{\mu_F}$ generates $\Ext(\mu_{F},\Z)$. 
In this case $E_{\mu_F}$ is free of rank one. Letting $\tilde x$ denote a generator, the extension is given explicitly by 
\begin{equation}
\xymatrix{E_{\mu_F}\colon 0\ar[r]&{\Z}\ar[r]^-{\iota_x}&{E_{\mu_F}}\ar[r]^-{\pi_x}&{\mu_F}\ar[r]&0},
\end{equation}
where $\iota_x$ takes $1$ to $\vert\mu_F\vert\tilde x$ and $\pi_x$ takes $\tilde x$ to a primitive root of unity $x$.\qed
\end{cor}


\begin{lemma}\label{primextwedge}
Suppose $E$ is $p$--primitive for an odd prime $p$. Then $\wedge^2(E_{\mu_{p^\infty}})=\Z/2\Z$ generated by $1\wedge 1$. If $E$ is also $2$--primitive, the map $\wedge^2(E_{\mu_{p^\infty}})\to\wedge^2(E)$ is $0$. Otherwise it is injective.
\end{lemma}
\begin{proof}
We assume that $E_{\mu_{p^\infty}}\cong \Z[1/p]$ leaving the simpler case $E_{\mu_{p^\infty}}\cong \Z$ to the reader. It is easy to see that $\Z[1/p]$ is $2$--torsion generated by elements $p^{-k}\wedge p^{-k}$, and since $p$ is odd, $p^{-k}\wedge p^{-k}=p^{2k}(p^{-k}\wedge p^{-k})=1\wedge 1$. By Corollary~\ref{Exteq}, $1$ is $2$--divisible in $E$ if and only if $E$ is $2$--primitive. This concludes the proof.
\end{proof}
Note that a primitive extension is $2$--primitive if and only if the characteristic of $F$ is not $2$. This is because $\mu_2$ is trivial in characteristic $2$ and non-trivial otherwise. 

\begin{prop}\label{Emuexactsequence} Let $E$ be a primitive extension and let $E(\mu_F)=2E_{\mu_F}$ if the characteristic of $F$ is $2$ and $E(\mu_F)=E_{\mu_F}$ otherwise. We have an exact sequence
\begin{equation}
\xymatrix{0\ar[r]&{E(\mu_F)}\ar[r]^-\iota & {E}\ar[r]^-\beta &{\wedge^2(E)}\ar[r]^-{\pi\wedge\pi}& {\wedge^2(F^*)}\ar[r]& 0,}
\end{equation}
where $\beta$ is the map taking $e$ to $e\wedge 1$. 
\end{prop}

\begin{proof}
The only non-trivial part is to show that the kernel of $\beta$ is $E(\mu_F)$.
First note that if $e\wedge 1=0\in\wedge^2(E)$, $e$ and $1$ must be linearly dependent over $\Z$ (this follows from Lemma~\ref{zeroinwedge}). Hence, $e$ must be in $E_{\mu_F}$. 

Let $e\in E_{\mu_F}$. We may assume that $e\in E_{\mu_{p^\infty}}$ for some prime $p$. 
Suppose the characteristic of $F$ is not $2$. By Lemma~\ref{primextwedge}, the restriction of $\beta$ to $E_{\mu_{p^\infty}}$ is zero if $p\neq 2$, and by inspection, this also holds if $p=2$. Hence, $\beta(e)=0$. Now suppose the characteristic of $F$ is $2$. In particular $p$ must be odd. Pick integers $k$ and $l$ such that $p^ke=l\in E_{\mu_{p^\infty}}$. By Lemma~\ref{primextwedge}, $e\wedge 1=p^ke\wedge 1=l\wedge 1$ is zero in $\wedge^2(E)$ if and only if $l$ is even. Hence, $p^k e$, and therefore also $e$, is $2$--divisible.
\end{proof}

\begin{cor}\label{BhatEandB} Let $E$ be primitive. There is an exact sequence
\[\xymatrix{{\widetilde{\mu_F}}\ar[r]^-\chi&{\widehat\B_E(F)}\ar[r]^-\pi&{\B(F)}\ar[r]&0,}\]
where $\widetilde{\mu_F}=E(\mu_F)/2\Z$.\qed
\end{cor}


Note that $\widetilde{\mu_F}$ is independent of $E$ up to isomorphism. If the characteristic of $F$ is not $2$, $\widetilde{\mu_F}$ is the unique non-trivial $\Z/2\Z$--extension of $\mu_F$, and in characteristic $2$, $\widetilde{\mu_F}$ is just $\mu_F$. The notation $\widetilde{\mu_F}$ thus agrees with that of Suslin~\cite{Suslin}.

\section{The extended Bloch group of a field}\label{BhatF}
We now show that if we impose some conditions on $F$, the extended Bloch groups $\widehat \B_E(F)$ are naturally isomorphic, and we can glue them together to form an extended Bloch group $\widehat \B(F)$ depending only on $F$. This group admits a natural action by the automorphism group of $F$.

\begin{defn}\label{freefielddefn} A field $F$ is called \emph{free} if it satisfies the following two conditions: 
\begin{enumerate}[(i)]
\item $F^*/\mu_F$ is a free abelian group;\label{field1}
\item $\vert \mu_F\vert<\infty$.\label{field2}
\end{enumerate}
\end{defn}
Free fields include number fields and finite fields, and freeness is preserved under finite transcendental extensions. For a discussion of fields satisfying \eqref{field1}, see e.g.~May~\cite{May}. 
Throughout this section $F$ denotes a free field.

Corollary~\ref{cond2cor} implies that primitive $\Z$--extensions of $F^*$ are in one-one correspondence with primitive roots of unity. If $x$ is a primitive root of unity, we let $E_x$ denote the corresponding primitive extension. The restriction of $E_x$ to $\mu_F$ is free of rank one with generator $\tilde x$ mapping to $x\in F^*$. Since $F^*$ is free modulo torsion, any extension of $F^*$ is uniquely determined by its restriction to $\mu_F$. The following result is a direct consequence of this discussion and Theorem~\ref{isoonBhat}.

\begin{lemma} Let $F'$ be any field (not necessarily free) and let $\sigma\colon F\to F'$ be an embedding. Let $E'$ be a $\Z$--extension of $F'^*$ and let $E$ be a primitive $\Z$--extension of $F^*$. There exists a covering $\widehat\sigma\colon E\to E'$ of $\sigma$. The induced map $\widehat\sigma_*\colon\widehat\B_{E}(F)\to\widehat\B_{E'}(F')$ depends only on $\sigma$ and not on the choice of covering.\qed
\end{lemma}


\begin{cor}\label{extmap}
For any pair $E_x$, $E_y$ of primitive $\Z$--extensions of $F^*$, there exists a covering $\Psi_{xy}\colon E_x\to E_y$ of the identity on $F$. The induced map $\Psi_{xy*}\colon\widehat\B_{E_x}(F)\to\widehat\B_{E_y}(F)$ is an isomorphism with inverse $\Psi_{yx*}$.
\end{cor}
\begin{proof}
Existence of $\Psi_{xy}$ follows from Lemma~\ref{extmap}, and since the coverings $\Psi_{xy}\circ\Psi_{yx}$ and $\Psi_{yx}\circ\Psi_{xy}$ are equivalent to the identities on $E_y$ and $E_x$, the result follows.
\end{proof}

\begin{cor}\label{embaction}
Let $\tau$ be an automorphism of $F$. For each primitive root of unity $x\in F$, there exists a unique covering $\widehat\tau_x\colon E_x\to E_{\tau(x)}$ of $\tau$. The induced maps $\widehat\tau_{x*}\colon \widehat \B_{E_x}(F)\to\widehat\B_{E_{\tau(x)}}(F)$ satisfy 
\begin{equation}\label{Psitaueq}\widehat\tau_{y*}\circ\Psi_{xy*}=\Psi_{\tau(x)\tau(y)*}\circ\widehat\tau_{x*}.\end{equation}
\end{cor}
\begin{proof}
Existence of $\widehat\tau_x$ follows from Lemma~\ref{extmap}, and since $\widehat\tau_{y}\circ\Psi_{xy}$ and $\Psi_{\tau(x)\tau(y)}\circ\widehat\tau_{x}$ are both coverings of $\tau$, the result follows.
\end{proof}



\begin{defn}\label{extblochoffield}
The extended Bloch group of $F$ is defined by
\begin{equation}
\widehat\B(F)=\lim_{\longleftarrow} \widehat\B_{E_x}(F)=\left\{(\alpha_{E_x})\in\prod\widehat\B_{E_x}(F)\bigm\vert\alpha_{E_y}=\Psi_{xy*}(\alpha_{E_x})\right\},
\end{equation}
where the product is over primitive roots of unity.
\end{defn}

\begin{remark}
If $F$ is not free, we can still define $\widehat\B(F)$ as an inverse limit, but in the general case, the primitive extensions do not form a directed set, and we do not see how to establish the desired connection with the classical Bloch group.
\end{remark}

\begin{prop}\label{Galoisaction}
There is a natural action of $\Aut(F)$ on $\widehat\B(F)$, where each automorphism acts by an isomorphism.
\end{prop}
\begin{proof} This follows directly from \eqref{Psitaueq}.
\end{proof}

\begin{prop}\label{almostexact}
There is an exact sequence
\[\xymatrix{{\widetilde{\mu_{F}}}\ar[r]^-\chi&{\widehat\B(F)}\ar[r]^-\pi&{\B(F)}\ar[r]&0.}\]
\end{prop}
\begin{proof}
This is an easy consequence of Corollary~\ref{BhatEandB}.
\end{proof}

Note that the action of $\Aut(F)$ on $\widetilde{\mu_F}$ is through the quadratic character.


\subsection{Embeddings in $\C$ and the regulator}


Recall that $\widehat\B(\C)$ is the extended Bloch group associated to the extension of $\C^*$ given by the exponential map.  

Let $\sigma$ be an embedding of $F$ in $\C$. By Lemma~\ref{extmap}, each primitive extension $E_x$ admits a covering $\widehat\sigma_x\colon E_x\to \C$ of $\sigma$. 
\begin{lemma}\label{primextinj} The induced map $\widehat\sigma_{x*}\colon\widehat\Pre_{E_x}(F)\to\widehat\Pre(\C)$ restricts to an injection $E_x/2\Z\to\C/4\pi i\Z$. 
\end{lemma}
\begin{proof}
Clearly, $\widehat\sigma_x\colon E_x\to\C$ is injective. Since $\widehat\sigma_x$ must take $1$ to $2\pi i k$, where $k$ is relatively prime to $\vert\mu_F\vert$, the result follows from Lemma~\ref{ktoksq}.
\end{proof}

\begin{cor}\label{primextinjcor}
Let $F$ be a free field admitting an embedding in $\C$. The map $\chi\colon E/2\Z\to\widehat\Pre_E(F)$ is injective for all primitive extensions $E$.
\end{cor}
\begin{proof}
We may assume that $E=E_x$. It is enough to prove that $\widehat\sigma_{x*}\circ\chi\colon E_x/2\Z\to\widehat\Pre(\C)$ is injective. By Lemma~\ref{ktoksq}, $\widehat\sigma_{x*}\circ\chi=\chi_\C\circ\widehat\sigma_{x*}$, where $\chi_\C$ denotes the map $\chi\colon\C/4\pi i\Z\to\widehat\Pre(\C)$. By Corollary~\ref{primextinjcor} (and Remark~\ref{CstarandEmod2}), this is a composition of injective maps, hence injective.
\end{proof}
Since $\widehat\sigma_x$ and $\widehat\sigma_y\circ\Psi_{xy}$ both cover $\sigma$, the induced maps satisfy $\widehat\sigma_{x*}=\widehat\sigma_{y*}\circ\Psi_{xy*}$, and we obtain a map
\begin{equation}
\sigma_*\colon\widehat\B(F)\to\widehat\B(\C)
\end{equation}
depending only on $\sigma$. The following is a simple consequence of Lemma~\ref{primextinj} and Corollary~\ref{primextinjcor}.
\begin{prop}\label{injectiveontor}
Let $F$ be a free field admitting an embedding $\sigma$ in $\C$. The map $\sigma_*\colon\widehat\B(F)\to\widehat\B(\C)$ restricts to an injection $\widetilde{\mu_F}\to\mu_\C$. Furthermore, the sequence
\begin{equation}\label{exactforfreefields} 
\xymatrix{0\ar[r]&{\widetilde{\mu_F}}\ar[r]^-\chi&{\widehat\B(F)}\ar[r]^-\pi&{\B(F)}\ar[r]&0.}
\end{equation}
is exact.\qed
\end{prop}

\subsubsection{The regulator} A number field $F$ of type $[r_1,r_2]$ has $r_1$ real embeddings and $r_2$ pairs of complex embeddings. 
The regulator map~\eqref{regulator} is equivariant with respect to 
the action on $\widehat\B(\C)$ by complex conjugation given by $[z;p,q]\mapsto [\bar z;-p,-q]$.
We therefore obtain a regulator 
\begin{equation}\label{regonBhat}\widehat\B(F)\to (\R/4\pi^2\Z)^{r_1}\oplus (\C/4\pi^2\Z)^{r_2}.\end{equation}
\begin{ex}\label{Rexample} Let $x$ be a root of $p(x)=x^4-x^3+2x^2-2x+1$, and consider the number field $F=\Q(x)$. One can check that $\mu_F$ has order $6$ and is generated by $w=x^3+x$. Let $E$ be the primitive extension corresponding to $w$ and let $\tilde w$ be the generator of $E_{\mu_F}$. Let $u=-x^3-2x+1$ and let $v=x^2-x+1$. The relations
\[1-u=u^2 w^4,\qquad v=w^3 u^{-2},\qquad 1-v=u^{-3} w,\]
are easily verified, and it follows that $\alpha=[u]+2[v]\in\B(F)$. Let $\tilde u$ be a lift of $u$, and consider the lift
\begin{equation}\tilde\alpha=(\tilde u,2\tilde u+4\tilde w)+2(-2\tilde u+3\tilde w,-3\tilde u+\tilde w)-3\chi(\tilde u)\in\widehat \Pre_E(F)\end{equation}
of $\alpha$. One easily checks that $\widehat\nu(\tilde \alpha)=0$, so $\tilde \alpha$ is in $\widehat \B_E(F)$. One can check, e.g.~using Lemma~\ref{Rformula}, that $\tilde\alpha$ is independent of the particular choice of $\tilde u$. 

Let $z$ be the root of $p$ given by $z=-0.1217\ldots + i1.3066\ldots$, and let $\sigma$ denote the corresponding embedding. 
Then $\sigma(w)=\exp(-\pi i/3)$. Let $\widehat\sigma\colon E\to\C$ be a covering of $\sigma$. Letting $\log$ denote the principal branch of logarithm, we may assume that 
\[\widehat\sigma(\tilde w)=-\pi i/3,\qquad\widehat \sigma(\tilde u)=\log(\sigma(u))=-0.2717\ldots-i 0.6165\ldots.\] 
Using Lemma~\ref{ktoksq}, we see that $\widehat\sigma_*$ takes $\chi(\tilde u)$ to $-\chi(\log(\sigma(u)))$. We now have
\begin{multline*}\sigma_{*}(\alpha)=(-0.2717\ldots-i 0.6165\ldots,-0.5435\ldots-i 5.4218\ldots)\\+2(0.5435-i 1.9085\ldots,0.8153\ldots+i 0.8023\ldots)\\+3\chi(-0.2717\ldots-i 0.6165\ldots).\end{multline*}
Using \eqref{regulator} or \eqref{Zagierregulator}, (and Remark~\ref{CstarandEmod2}), we obtain
\begin{equation*}
R(\sigma_{*}(\alpha))= -7.4532\ldots-i 2.3126\ldots\in \C/4\pi^2\Z. 
\end{equation*}
\end{ex}

\begin{remark} Examples like the above can be produced in abundance using computer software like PARI/GP.
\end{remark}

\section{The other version of the extended Bloch group}\label{BhatPSL}
As mentioned in the introduction there are two versions, $\widehat\B(\C)_{\SL}$ and $\widehat\B(\C)_{\PSL}$, of the extended Bloch group. They are isomorphic to $H_3(\SL(2,\C))$ and $H_3(\PSL(2,\C))$, respectively. In this section we define the algebraic version of $\widehat\B(\C)_{\PSL}$, and discuss its relationship with hyperbolic geometry.
We stress that this version is only defined when the extension $E$ of $F$ is $2$--primitive.

Let $F$ be a field and let $E$ be a $2$--primitive extension of $F^*$ by $\Z$. By Corollary~\ref{Exteq}, $1\in\Z\subset E$ is uniquely two-divisible, so $\frac{1}{2}\in E$ is well defined. Consider the set of \emph{odd flattenings}
\begin{equation}
\overline{F}_E=\big\{(e,f)\in E\times E\bigm\vert\pm\pi(e)\pm\pi(f)=1\in F\big\},
\end{equation}
Since knowledge of $z$ and $1-z$ up to a sign determines $z$, we have a map $\pi\colon\overline F_E\to F\setminus\{0,1\}$.
Define $\overline{\FT}_E$ as in Definition~\ref{fiveeqdef}, and
define $\widehat{\Pre}_E(F)_{\PSL}$ to be the abelian group generated by $\overline{F}_E$ subject to the relations
\begin{gather*}
\sum_{i=0}^4(-1)^i(e_i,f_i)=0\text{ for }\big((e_0,f_0),\dots(e_4,f_4)\big)\in\overline{\FT}_E\\
(e+\frac{1}{2},f+\frac{1}{2})+(e,f)=(e+\frac{1}{2},f)+(e,f+\frac{1}{2}).
\end{gather*}
The second relation is the analog of the \emph{transfer relation}; see Goette--Zickert~\cite{GZ} or Neumann~\cite[Proposition~7.2]{Neumann}. 
The extended Bloch group $\widehat\B_E(F)_{\PSL}$ is defined as in Definition~\ref{extblochoffield}.
Note that $\pi\colon\overline F_E\to F\setminus\{0,1\}$ induces maps from the extended groups $\widehat{\Pre}_E(F)_{\PSL}$ and $\widehat\B_E(F)_{\PSL}$ to the classical ones.

For $a\in E\setminus\Z$ let $\overline\chi(a)=(e,f+1/2)-(e,f)$, where $(e,f)$ is any flattening of $\pi(a)$. The analog of Lemma~\ref{GZ3.3} holds, proving independence of $f$, and independence of $e$ follows from the transfer relation. 
As in Corollary~\ref{extofchi}, $\overline\chi$ extends to a homomorphism $\overline\chi\colon E\to\widehat{\Pre}_E(F)_{\PSL}$, and a computation as in \eqref{2chi0} shows that $\overline\chi(1)=2\overline\chi(\frac{1}{2})=0$.

\begin{lemma}\label{SLPSL}There is a commutative diagram of exact sequences.
\begin{equation}
\cxymatrix{{{E/2\Z}\ar[r]^-\chi\ar[d]^2&{\widehat{\Pre}_E(F)}\ar[d]^p\ar[r]&{\Pre(F)}\ar[r]\ar@{=}[d]&0\\{E/\Z}\ar[r]^-{\overline\chi}&{\widehat\Pre_E(F)_{\PSL}}\ar[r]&{\Pre(F)}\ar[r]&0,}}
\end{equation}
\end{lemma}

\begin{proof} Exactness of the bottom sequence is proved as in Theorem~\ref{preblochrel}. For commutativity of the left diagram, note that in $\widehat\Pre_E(F)_{\PSL}$, we have
\begin{equation} (e,f+1)-(e,f)=(e,f+1)-(e,f+\frac{1}{2})+(e,f+\frac{1}{2})-(e,f)=2\chi(e)=\chi(2e).\end{equation}
This proves the result.
\end{proof}

It follows from Proposition~\ref{Emuexactsequence} that we have an exact sequence
\begin{equation}
\xymatrix{{\mu_F}\ar[r]^-{\overline\chi}&{\widehat\B_E(F)_{\PSL}}\ar[r]&{\B(F)}\ar[r]&0}
\end{equation} 

If $F$ is a free field, we can form $\widehat\B(F)_{\PSL}$ as in Section~\ref{BhatF}. 
An embedding $F\to\C$ induces a map $\widehat\B(F)_{\PSL}\to\widehat\B(\C)_{\PSL}$ restricting to an injection $\mu_F\to\mu_{\C}$. In particular, $\overline\chi$ is injective, and it follows from Lemma~\ref{SLPSL} that there is an exact sequence
\begin{equation}\label{hatandbar}
\xymatrix{0\ar[r]&{\Z/4\Z}\ar[r]&{\widehat\B(F)}\ar[r]^-p&{\widehat{\B}(F)_{\PSL}}\ar[r]&{\Z/2\Z}\ar[r]&0.}
\end{equation}

\begin{remark}\label{PSLBhatlift} One can concretely determine if an element $\alpha\in\widehat\B_E(F)_{\PSL}$ lifts: Pick any element $\tau\in\widehat\Pre_E(F)$ lifting $\pi(\alpha)\in\B(F)$. Then $x=\alpha-p\tau\in\widehat\Pre_E(F)_{\PSL}$ is in $E/\Z=F^*$, and it follows from Lemma~\ref{SLPSL} that $\alpha$ lifts if and only if $x$ is a square in $F^*$.
\end{remark}

\subsection{Parabolic representations}
The result below generalizes the main result in Zickert~\cite{Zickert}.
\begin{thm}\label{Zickertthm}
Let $M$ be a tame manifold, and let $F$ be a free field. A parabolic representation $\rho\colon\pi_1(M)\to\PSL(2,F)$ defines a fundamental class $[\rho]\in \widehat\B(F)_{\PSL}$. If $\sigma\colon F\to \C$ is an embedding, $\sigma_*([\rho])=[\sigma\circ\rho]\in \widehat\B(\C)_{\PSL}$.
\end{thm}
\begin{proof}
In Zickert~\cite{Zickert} we defined a homomorphism $H_3(\PSL(2,\C),P)\to\widehat\B(\C)_{\PSL}$, where $P=\{\left(\begin{smallmatrix}1&*\\0&1\end{smallmatrix}\right)\}$. This map can be defined over any free field by replacing the logarithms in Zickert~\cite[(3.6)]{Zickert} by a lift of $E\to F^*$. By Zickert~\cite[Theorem~5.13]{Zickert} a decorated parabolic representation $\rho\colon\pi_1(M)\to\PSL(2,F)$ of $M$ has a fundamental class $H_3(\PSL(2,F),P)$, and as in Zickert~\cite[Theorem~6.10]{Zickert}, the image of the fundamental class in $\widehat\B(F)_{\PSL}$ is independent of the decoration. This proves the first statement. The second statement is an immediate consequence of the definition.
\end{proof}

\begin{remark}
We stress that $\widehat\B(F)_{\PSL}$ is \emph{not} isomorphic to $H_3(\PSL(2,F))$ in general. 
\end{remark}

\section{Ideal cochains and flattenings of $3$--cycles}\label{FlatSection}
Fix a field $F$ and a primitive extension $E$ of $F^*$ by $\Z$. By a \emph{simplex} we will always mean a standard simplex together with a fixed vertex ordering. Unless otherwise specified, a simplex means a $3$--simplex.

\begin{defn}\label{flatsimplex}
An \emph{(algebraic) flattening} of a simplex $\Delta$ is an association of an algebraic flattening $(e,f)\in \widehat F_E$ to $\Delta$. If $(e,f)$ is a flattening of $z\in F\setminus\{0,1\}$, we refer to $z$ as the \emph{cross-ratio} of the flattened simplex.
\end{defn}

\begin{remark}
Definition~\ref{flatsimplex} is a generalization of \emph{even flattenings}, i.e.~flattenings $[z;p,q]$, with $p$ and $q$ even. Neumann~\cite{Neumann} also allows odd values of $p$ and $q$.
\end{remark}

We will associate elements in $E$ to edges of a flattened simplex as indicated in Figure~\ref{logpar}. We will refer to these elements as \emph{log-parameters}.

\begin{defn}\label{3cycledef}
An \emph{(ordered, oriented) $3$--cycle} is a space $K$ obtained from a collection of simplices by gluing together pairs of faces using simplicial attaching maps preserving the vertex orderings. If all faces have been glued, we say that $K$ is \emph{closed}. We assume that the manifold $K_0$ with boundary (and corners) obtained by removing disjoint regular neighborhoods of the $0$--cells is oriented. If $\Delta_i$ is a simplex in $K$, we let $\varepsilon_i$ be a sign encoding whether or not the orientation of $\Delta_i$ coming from the vertex ordering agrees with the orientation of $K_0$. 
\end{defn}
\begin{remark}Neumann~\cite{Neumann} only considers closed $3$--cycles. With our definition, a single simplex is a $3$--cycle.\end{remark}


The definition below is the analog of Neumann~\cite[Definition 4.4]{Neumann}, which the reader may consult for further details.
\begin{defn}
Let $K$ be a closed $3$--cycle. A \emph{flattening} of $K$ is a choice of flattening of each simplex of $K$ such that the total log-parameter (summed according to the sign conventions of Neumann~\cite[Definition~4.3]{Neumann})
around each edge is zero. If the total log-parameter along any normal curve in the star of each zero-cell is zero, it is called a \emph{strong flattening}.
\end{defn}

\begin{remark}\label{noparity}
We do not need any conditions on the parity. 
This is because the parity condition is automatically satisfied for even flattenings. The proof of this fact is identical to the proof of Neumann~\cite[Proposition 5.3]{Neumann}. 
\end{remark}

If $K$ is a $3$--cycle, and $G$ is an abelian group, we let $C^1(K;G)$ denote the set of cellular $1$--cochains in $K$ with values in $G$. A cochain $c\in C^1(K;G)$ naturally associates to each edge of each simplex of $K$ an element in $G$. Edges that are identified in $K$ acquire the same labeling. If $\Delta$ is a simplex of $K$, we let $c_{ij}(\Delta)$ denote the labeling of the edge joining the vertices $i$ and $j$ in $\Delta$, see Figure~\ref{edgepar}. 

\begin{figure}[ht]
\centering
\begin{minipage}[t]{0.45\textwidth}
\centering
\includegraphics[width=4.8cm]{K3Bhat.1}
\caption{Associating log-parameters to edges of a flattened simplex.}\label{logpar}
\end{minipage}
\quad
\begin{minipage}[t]{0.4\textwidth}
\centering
\includegraphics[width=4.8cm]{K3Bhat.2}
\caption{Edge labelings arising from a cochain.}\label{edgepar}
\end{minipage}
\end{figure}

\begin{defn}\label{idealcochain} Let $K$ be a $3$--cycle. A cochain $c\in C^1(K;F^*)$ is called an \emph{ideal cochain} if for each simplex $\Delta_i$ in $K$, there is an element $z_i\in F^*\setminus\{0,1\}$, such that the associated labeling of edges satisfies 
\begin{equation}\label{cij}
\frac{c_{03}^ic_{12}^i}{c_{02}^ic_{13}^i}=z_i,\qquad \frac{c_{01}^ic_{23}^i}{c_{02}^ic_{13}^i}=1-z_i.
\end{equation}
An ideal cochain thus associates cross-ratios to each simplex. 
\end{defn}

\begin{remark} Not all $3$--cycles admit ideal cochains. The fact that $3$--cycles admitting ideal cochains exist follows from Remark~\ref{idealcochainsexist} below.
\end{remark}

We wish to prove that each lift $\tilde c\in C^1(K;E)$ of an ideal cochain $c$ determines an element in $\widehat\sigma(\tilde c)\in\widehat\Pre_E(F)$ such that if $K$ is closed, $\widehat\sigma(\tilde c)$ is in $\widehat\B_E(F)$ and is independent of the choice of lift. In other words, an ideal cochain on a closed $3$--cycle determines an element in $\widehat\B_E(F)$.

Let $I_n$ be the free abelian group on cochains $\tilde c\in C^1(\Delta^n;E)$ on an $n$--simplex $\Delta^n$, whose restriction to each $3$--dimensional face is the lift of an ideal cochain.
The usual boundary map induces boundary maps $\partial\colon I_n\to I_{n-1}$, making $I_*$ into a chain complex.
A lift $\tilde c$ of an ideal cochain $c$ on $K$ determines an element in $I_3$ given by $\sum\varepsilon_i\tilde c^i$. We may thus regard lifts of ideal cochains as elements in $I_3$. Note that if $K$ is closed, $\tilde c$ is a cycle, i.e.~$\partial \tilde c=0\in I_2$.

Consider the maps $\widehat\sigma\colon I_3\to \Z[\widehat F_E]$ and $\mu\colon I_2\to\wedge^2(E)$ defined on generators by 
\begin{gather}
\label{cijflat}
\widehat\sigma(\tilde c)=(\tilde c_{03}+\tilde c_{12}-\tilde c_{02}-\tilde c_{13},\tilde c_{01}+\tilde c_{23}-\tilde c_{02}-\tilde c_{13}),\\
\mu(\tilde c)=-\tilde c_{01}\wedge\tilde c_{02}+\tilde c_{01}\wedge\tilde c_{12}-\tilde c_{02}\wedge\tilde c_{12}+\tilde c_{02}\wedge\tilde c_{02}.
\end{gather}

\begin{lemma}\label{I3Bhat}
There is a commutative diagram 
\begin{equation}
\cxymatrix{{I_4\ar[r]^-\partial\ar[d]^-{\widehat\sigma}&I_3\ar[r]^-\partial\ar[d]^-{\widehat\sigma}&I_2\ar[d]^-\mu\\{\Z[\widehat{\FT}_E]}\ar[r]^-{\widehat\rho}&{\Z[\widehat F_E]}\ar[r]^-{\widehat\nu}&\wedge^2(E).}}
\end{equation}

\end{lemma}
\begin{proof}
Let $\tilde c\in I_4$ be a generator, and suppose $\widehat\sigma\circ\partial(\tilde c)=\sum (-1)^i(e_i,f_i)$. We must prove that the flattenings $(e_i,f_i)$ satisfy the five equations of Definition~\ref{fiveeqdef}. 
We check the first equation, and leave the verification of the four others to the reader. By~\eqref{cijflat} we have
\begin{gather*}
e_0=\tilde c_{14}+\tilde c_{23}-\tilde c_{13}-\tilde c_{24}\\
e_1=\tilde c_{04}+\tilde c_{23}-\tilde c_{03}-\tilde c_{24}\\
e_2=\tilde c_{04}+\tilde c_{13}-\tilde c_{14}-\tilde c_{03},
\end{gather*}
and it follows that $e_2=e_1-e_0$. 

Letting $\tilde c\in I_3$ be a generator, we see that $\widehat\nu\circ\widehat\sigma(\tilde c)$ and $\mu\circ\partial(\tilde c)$ are both a sum of $12$ terms of the form $\tilde c_{ij}\wedge \tilde c_{kl}$ with $\{i,j\}\neq\{k,l\}$ and $4$ two-torsion terms summing to $\tilde c_{02}\wedge\tilde c_{02}+\tilde c_{13}\wedge\tilde c_{13}$ (the two other terms cancel out). The two-torsion terms thus match up, and one easily checks that the other terms match up as well, proving commutativity of the right square.
\end{proof}

\begin{cor} If $K$ is closed and $\tilde c$ is a lift of an ideal cochain on $K$, $\widehat\sigma(\tilde c)$ is in $\widehat\B_E(F)$.\qed
\end{cor}

\begin{prop}\label{strongflat}
Let $K$ be a closed $3$--cycle. The set of flattenings coming from a lift of an ideal cochain is a strong flattening of $K$.
\end{prop}
\begin{proof}
The proof is identical to the proof of Zickert~\cite[Theorem~6.5]{Zickert}, so we omit some details.
Consider a curve $\alpha$ in the star of a $0$--cell as shown in Figure~\ref{cfigure}. 
When $\alpha$ passes through a simplex, it picks up a log-parameter, which is a signed sum of four terms. The signs are shown in the figure.
If $\alpha$ is a closed curve, it is not difficult to see that all terms must cancel out.
\begin{figure}[ht]
\centering
\includegraphics[width=5.5cm]{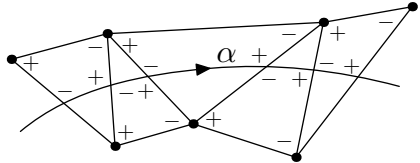}
\caption{A normal curve in the star of a $0$--cell. Each edge and each vertex corresponds to a $1$--cell in $K$.}\label{cfigure}
\end{figure}
\end{proof}

\begin{lemma}\label{sameeltlemma} Let $K$ be a $3$--cycle and let $c$ be an ideal cochain on $K$. Let $e$ be an interior $1$--cell of $K$, and let $\alpha_e\in C^1(K;\Z)$ be the cochain taking $e$ to $1$ and all other $1$--cells to $0$. For every lift $\tilde c$ of $c$, we have
\begin{equation}\label{alphaE}\widehat\sigma(\tilde c+\alpha_e)=\widehat\sigma(\tilde c)\in\widehat\Pre_E(F).\end{equation}
\end{lemma}

\begin{proof}
The map $\widehat\sigma$ associates flattenings to the simplices of $K$, and using \eqref{cijflat} one checks that the flattenings coming from $\tilde c$ and $\tilde c+\alpha$ differ by Neumann's cycle relation~\cite[Section~6]{Neumann} about $e$ (or rather the obvious generalization of this relation to algebraic flattenings). Neumann's proof that the cycle relation is a consequence of the lifted five term relation carries over to the algebraic setup word by word. 
\end{proof}
\begin{cor}\label{sameelement} If $K$ is a closed $3$--cycle, $\widehat\sigma(\tilde c)=\widehat\sigma(\tilde c+\alpha)\in\widehat\B_E(F)$ for any $\alpha\in C^1(K;\Z)$. Hence, an ideal cochain $c$ on $K$ determines an element in $\widehat\sigma(c)\in\widehat\B_E(F)$.\qed
\end{cor}

\subsection{The action of $Z^1(K;\Z/2\Z)$ on ideal cochains}\label{actionsection}
Let $K$ be a closed $3$--cycle, and suppose that the characteristic of $F$ is not $2$. The group $Z^1(K;\Z/2\Z)$ of cellular $1$--cocycles on $K$ acts on the set of ideal cochains by multiplication. Note that the action does not change the cross-ratios. A cochain $\alpha\in Z^1(K;\Z/2\Z)$ determines a map 
\[B\alpha\colon K\to B(\Z/2\Z)=\R P^{\infty},\]
and we wish to prove that the elements in $\widehat\B_E(F)$ associated to ideal cochains $c$ and $\alpha c$ differ by a two-torsion element which is zero if and only if $B\alpha_*([K])$ is zero in $H_3(\R P^\infty)=\Z/2\Z$.

The homology of a group $G$ is the homology of the complex $B_*(G)$ where $B_n(G)$ is generated by symbols $\langle g_1\vert\dots\vert g_n\rangle$ with $g_i\in G$. Such tuples are in one-one correspondence with $G$--cocycles on $\Delta^n$; a cocycle is uniquely given by its values on the edges between vertices $i$ and $i+1$. Under this correspondence, the boundary maps are induced by the standard ones.
Given a cocycle $\alpha\in Z^1(K;G)$ the restriction of $\alpha$ to $\Delta_i$ determines a tuple $\langle g_1^i\vert g_2^i\vert g_3^i\rangle$ and by Zickert~\cite[Proposition~5.7]{Zickert} we have 
\[B\alpha_*([K])=\sum\varepsilon_i\langle g_1^i\vert g_2^i\vert g_3^i\rangle.\]

Let $\alpha\in Z^1(\Delta^3;\Z/2\Z)$ and let $c\in C^1(\Delta^3;F^*)$ be an ideal cochain. Let $c'=\alpha c$, and pick a lift $\tilde c$ of $c$. Then $\tilde c$ endows $\Delta^3$ with a flattening $\widehat\sigma(\tilde c)$ given by 
\[(e,f)=(\tilde c_{03}+\tilde c_{12}-\tilde c_{02}-\tilde c_{13},\tilde c_{01}+\tilde c_{23}-\tilde c_{02}-\tilde c_{13}).\]
Let $w\in E$ be the sum of the log-parameters at the edges where $\alpha_{ij}=-1$. One easily checks that $w$ is always (uniquely) two-divisible, e.g.~if $\alpha=\langle -1\vert 1\vert -1\rangle$, $w=2e+2(-e+f)=2f$.
Let $\tilde c^\prime$ be the lift of $c^\prime$ defined by
\begin{equation}\label{ctildeprime}
\tilde c^\prime_{ij}=\tilde c_{ij}+
\begin{cases} \frac{1}{2} &\text{ if } \alpha_{ij}=-1\\
 0&\text{ otherwise.}
\end{cases}
\end{equation}

\begin{lemma}\label{wdelta} Let $\delta\in \Z\subset E$ be $1$ if $\alpha=\langle -1\vert -1\vert -1\rangle$ and $0$ otherwise. We have
\begin{equation}
\widehat\sigma(\tilde c^\prime)-\widehat\sigma(\tilde c)=\chi(\frac{1}{2}w)+\chi(\delta)\in\widehat\Pre_E(F).
\end{equation}
\end{lemma}
\begin{proof}
This is done case by case using Lemma~\ref{Rformula}. If e.g.~$\alpha=
\langle 1\vert -1\vert 1\rangle$, $w=-2e$ and we have
\[\widehat\sigma(\tilde c^\prime)-\widehat\sigma(\tilde c)=(e,f-1)-(e,f)=\chi(-e)=\chi(\frac{1}{2}w)+\chi(\delta).\]
The other seven cases are similar and left to the reader.
\end{proof}

\begin{thm}\label{Z2action} Let $\alpha\in Z^1(K;\Z/2\Z)$. For any ideal cochain $c$, $\widehat\sigma(\alpha c)-\widehat\sigma(c)$ is two-torsion in $\widehat\B_E(F)$, which is trivial if and only if $B\alpha_*([K])\in H_3(\R P^\infty)=\Z/2\Z$ is trivial.
\end{thm}
\begin{proof} 
Let $\tilde c$ be a lift of $c$ and let $\tilde c^\prime$ be the lift of $\alpha c$ defined by \eqref{ctildeprime}. For each simplex $\Delta_i$ of $K$, we have elements $w_i$ and $\delta_i$ as above. Since the flattening of $K$
defined by $\tilde c$ is strong, $\sum \varepsilon_i w_i=0\in E$, and since $E$ has no two-torsion, Lemma~\ref{wdelta} implies that 
\[\widehat\sigma(\tilde c^\prime)-\widehat\sigma(\tilde c)=\sum \varepsilon_i\chi(\delta_i).\]
One easily checks that the isomorphism $H_3(B\Z/2\Z)\cong\Z/2\Z$ is induced by the map $B_3(\Z/2\Z)\to\Z/2\Z$ taking $\langle -1\vert -1\vert -1\rangle$ to $1$ and all other generators to $0$. This proves the result.
\end{proof}

\subsection{Ideal cochains and homology of linear groups}\label{idealcochainSection}

In this section we define a canonical map $\widehat\lambda\colon H_3(\SL(2,F))\to \widehat \B_E(F)$. This is purely algebraic and follows Dupont--Zickert~\cite{DupontZickert}. 
We assume that $F$ is infinite.
 
Let $C_*(F^2)$ be the chain complex generated in dimension $n$ by $(n+1)$--tuples of vectors in $F^2\setminus\{0\}$ in general position, together with the usual boundary map. 
Letting $p\colon F^2\setminus\{0\}\to P^1(F)$ be the canonical projection, a simple computation (as in Dupont--Zickert~\cite[Section~3.1]{DupontZickert}) shows that 
\begin{equation}\label{detformula}
z=\frac{\det(v_0,v_3)\det(v_1,v_2)}{\det(v_0,v_2)\det(v_1,v_3)},\quad 1-z=\frac{\det(v_0,v_1)\det(v_2,v_3)}{\det(v_0,v_2)\det(v_1,v_3)},\end{equation}
where $z$ is the cross-ratio of the tuple $(p(v_0),p(v_1),p(v_2),p(v_3))$. It follows that there is a chain map
\begin{equation}\label{defofGamma}
\Gamma\colon C_*(F^2)\to I_*,\qquad \Gamma(v_0,\dots,v_n)_{ij}=\log\det(v_i,v_j).
\end{equation}
Here $\log$ denotes a fixed section of $\pi\colon E\to F^*$. We refer to it as a \emph{logarithm}.
Let $\widehat\lambda=\widehat\sigma\circ\Gamma$. Then $\widehat\lambda$ is $\SL(2,F)$--invariant, and by Lemma~\ref{I3Bhat} it induces a map $H_3(C_*(F^2)_{\SL(2,F)})\to\widehat\B_E(F)$. 

Recall that the homology of a group $G$ is the homology of $(F_*)_G=F_*\otimes_{\Z[G]}\Z$, where $F_*$ is any free resolution of $\Z$ by $G$-modules. One such resolution is the complex $C_*(G)$ of tuples in $G$. Note that $C_*(G)_G$ equals the complex $B_*(G)$ considered in Section~\ref{actionsection}. If $G=\SL(2,F)$ we may assume (see e.g.~Dupont--Zickert~\cite[Section~3.2]{DupontZickert}) that all tuples $(g_0,\dots g_n)$ are in general position in the sense that $(g_0v,\dots,g_nv)\in C_n(F^2)$ for some fixed $v\neq 0\in F^2$ (the particular choice is inessential). It follows that $\lambda$ canonically extends to a map
\begin{equation}\label{giflat}
\widehat\lambda\colon H_3(\SL(2,F))\to\widehat\B_E(F), \quad (g_0,g_1,g_2,g_3)\mapsto\widehat\sigma\circ\Gamma(g_0v,g_1v,g_2v,g_3v).
\end{equation}

\begin{remark}\label{idealcochainsexist}
Note that for each $\alpha\in H_3(\SL(2,F))$, $\widehat\lambda(\alpha)$ is induced by an ideal cochain on a $3$--cycle.
The fact that $\widehat\lambda$ is independent of the choice of logarithm follows from Corollary~\ref{sameelement}.
\end{remark}

\begin{remark}\label{GL2}
The map $\pi\circ\widehat\lambda\colon C_3(F^2)\to\Pre(F)$ is $\GL(2,F)$--invariant, and it follows that there is an induced map $H_3(\GL(2,F))\to\B(F)$. This map factors through $H_3(C_*(P^1(F))_{\GL(2,F)})$, and thus agrees with that of Suslin~\cite{Suslin}.
\end{remark}

\section{The extended Bloch group and algebraic $K$--theory}\label{AlgKeqBhat}
In this section, we prove our main result Theorem~\ref{mainthm}. We do this in three steps.
The first and most difficult step is to extend the map $\widehat\lambda$ from Section~\ref{idealcochainSection} to a map
$\widehat\lambda\colon H_3(\GL(3,F))\to \widehat\B_E(F).$
Once this has been done, we obtain a map
\begin{equation}\label{MaponK}\xymatrix{K_3(F)\ar[r]^-H& {H_3(\GL(F))}\ar[r]^-\cong &{H_3(\GL(3,F))}\ar[r]^-{\widehat\lambda}&{\widehat\B_E(F)}},\end{equation}
where $H$ is the Hurewicz map and the middle isomorphism is Suslin's stability result Theorem~\ref{Suslinstability}. In the first step, we only require that $F$ be infinite and that $E$ be primitive.
The second step is to prove that this map takes the image of $K_3^M(F)$ to $0$. To do this, we need to assume that $F$ is a number field, or more generally, a free field. The third and final step is to show that $\widehat\lambda$ induces a map between the diagrams \eqref{Suslinexact} and \eqref{exactforfreefields}. The result then follows from the five-lemma. 

\subsection{Step one: Extension of $\widehat \lambda$ to $H_3(\GL(3,F))$}\label{H3SL3} 
We start by constructing $\widehat\lambda$ on $H_3(\SL(2,F))$.
We assume that $F$ is infinite and that $E$ is primitive.
Much of the construction draws inspiration from Igusa~\cite{Igusa}, Fock--Goncharov~\cite{FockGoncharov}, and private discussions with Dylan Thurston and Stavros Garoufalidis. In Garoufalidis--Thurston--Zickert~\cite{GaroufalidisThurstonZickert} we generalize to $\SL(n,F)$ and discuss some of the underlying geometric ideas motivating the construction.


In Section~\ref{idealcochainSection} we associated an ideal cochain to a quadruple of vectors in $F^2$. We now generalize this to tuples of vectors in $F^3$. 
Let $\mathfrak v=(v_0,\dots,v_n)$ be a tuple of vectors in $F^3$ in general position, and let $w\in F^3$ be in general position with respect to the $v_i$'s. For each $i\in\{0,\dots,n+1\}$ we have a cochain $\tilde c^i_w\in C^1(\Delta^n;E)$ given by \begin{equation}\label{alphadets}
\tilde c_w^i(\mathfrak v)_{jk}=\begin{cases}\log\det (w,v_j,v_k)\text{ if }i\leq j<k\\\log\det(v_j,w,v_k)\text{ if }j<i\leq k\\\log\det(v_j,v_k,w)\text{ if }j<k<i,\end{cases}
\end{equation}
where, as in Section~\ref{idealcochainSection}, $\log$ is a fixed section of $\pi\colon E\to F^*$. 

\begin{lemma} Each $\tilde c^i_w(\mathfrak v)$ is in $I_n$, and for each restriction to a $3$--dimensional face of $\Delta^n$, the cross-ratio is independent of $i$.
\end{lemma}
\begin{proof}
We may assume that $\mathfrak v=(v_0,v_1,v_2,v_3)$. Let $c^i_w(\mathfrak v)$ be the projection of $\tilde c^i_w(\mathfrak v)$ to $C^1(\Delta^3;F^*)$, and let $p\colon F^3\setminus\{w\}\to P(F^3/\langle w\rangle)$ denote the map induced by projection. By applying a linear transformation if necessary, we may assume that $w=(1,0,0)$, and identify $F^3/\langle w\rangle$ with $F^2$. It now follows from \eqref{detformula} that the cross-ratio $z$ of the tuple $(p(v_0),p(v_1),p(v_2),p(v_3))$ of elements in $P(F^3/\langle w\rangle)\approx P^1(F)$ satisfies
\[z=\frac{\det(w,v_0,v_3)\det(w,v_1,v_2)}{\det(w,v_0,v_2)\det(w,v_1,v_3)},\quad 1-z=\frac{\det(w,v_0,v_1)\det(w,v_2,v_3)}{\det(w,v_0,v_2)\det(w,v_1,v_3)}.\]
It follows that $c^0_w(\mathfrak v)$ is an ideal cochain with cross-ratio $z$. 
Since the expressions
\[\frac{c^i_w(\mathfrak v)_{03}c^i_w(\mathfrak v)_{12}}{c^i_w(\mathfrak v)_{02}c^i_w(\mathfrak v)_{13}}, 
\qquad \frac{c^i_w(\mathfrak v)_{01}c^i_w(\mathfrak v)_{23}}{c^i_w(\mathfrak v)_{02}c^i_w(\mathfrak v)_{13}}\]
are independent of $i$, it follows that the same is true for all the $c^i_w(\mathfrak v)$'s. 
\end{proof}

If $\mathfrak v=(v_0,v_1,v_2,v_3)$, we let $(v_0,v_1,v_2,v_3)^i_w$ denote the flattening $\widehat\sigma(\tilde c^i_w(\mathfrak v))$. The log-parameters are given by 
\begin{equation}\label{alphaef}
\begin{gathered}
e=\tilde c^i_w(\mathfrak v)_{03}+\tilde c^i_w(\mathfrak v)_{12}-\tilde c^i_w(\mathfrak v)_{02}-\tilde c^i_w(\mathfrak v)_{13}\\
f=\tilde c^i_w(\mathfrak v)_{01}+\tilde c^i_w(\mathfrak v)_{23}-\tilde c^i_w(\mathfrak v)_{02}-\tilde c^i_w(\mathfrak v)_{13}.
\end{gathered}
\end{equation}

\begin{lemma}\label{boundaryi}
The following formulas hold in $\widehat \Pre_E(F)$ (note the superscripts). 
\small{
\begin{gather*}
\begin{aligned}
(v_1,v_2,v_3,v_4)_w^0\!-\!(v_0,v_2,v_3,v_4)_w^0\!+\!(v_0,v_1,v_3,v_4)_w^0\!-\!(v_0,v_1,v_2,v_4)_w^0\!+\!(v_0,v_1,v_2,v_3)_w^0&\!=\!0\\
(v_1,v_2,v_3,v_4)_w^0\!-\!(v_0,v_2,v_3,v_4)_w^1\!+\!(v_0,v_1,v_3,v_4)_w^1\!-\!(v_0,v_1,v_2,v_4)_w^1\!+\!(v_0,v_1,v_2,v_3)_w^1&\!=\!0\\
(v_1,v_2,v_3,v_4)_w^1\!-\!(v_0,v_2,v_3,v_4)_w^1\!+\!(v_0,v_1,v_3,v_4)_w^2\!-\!(v_0,v_1,v_2,v_4)_w^2\!+\!(v_0,v_1,v_2,v_3)_w^2&\!=\!0\\
(v_1,v_2,v_3,v_4)_w^2\!-\!(v_0,v_2,v_3,v_4)_w^2\!+\!(v_0,v_1,v_3,v_4)_w^2\!-\!(v_0,v_1,v_2,v_4)_w^3\!+\!(v_0,v_1,v_2,v_3)_w^3&\!=\!0\\
(v_1,v_2,v_3,v_4)_w^3\!-\!(v_0,v_2,v_3,v_4)_w^3\!+\!(v_0,v_1,v_3,v_4)_w^3\!-\!(v_0,v_1,v_2,v_4)_w^3\!+\!(v_0,v_1,v_2,v_3)_w^4&\!=\!0.
\end{aligned}
\end{gather*}
}

\noindent\normalsize{We will refer to the left hand sides of the five equations above as \emph{boundaries}, and denote them by $\partial^i(v_0,v_1,v_2,v_3,v_4)_w$, $i\in\{0,\dots,4\}$.}

\end{lemma}
\begin{proof}
We will show that each of the boundaries corresponds to a lifted five term relation. To do this we must prove that the flattenings satisfy the five equations of Definition~\ref{fiveeqdef}. Suppose we wish to verify that $f_3=f_2-f_1$ for the boundary $\partial^2$. The relevant terms involved in this are $(v_0,v_1,v_2,v_4)_w^2$, $(v_0,v_1,v_3,v_4)_w^2$ and ($v_0,v_2,v_3,v_4)_w^1$. If we denote their flattenings by $(e_3,f_3)$, $(e_2,f_2)$ and $(e_1,f_1)$, it follows from \eqref{alphaef} and \eqref{alphadets} that
\begin{equation*}
\begin{gathered}
\begin{aligned}
f_3&=\log(v_0,v_1,w)+\log(w,v_2,v_4)-\log(v_0,w,v_2)-\log(v_1,w,v_4)\\
f_2&=\log(v_0,v_1,w)+\log(w,v_3,v_4)-\log(v_0,w,v_3)-\log(v_1,w,v_4)\\
f_1&=\log(v_0,w,v_2)+\log(w,v_3,v_4)-\log(v_0,w,v_3)-\log(w,v_2,v_4),
\end{aligned}
\end{gathered}
\end{equation*}
where $\log(u,v,w)$ denotes $\log(\det(u,v,w))$. Hence, $f_3=f_2-f_1$ as desired. The verification of the other formulas are similar and are thus left to the reader.
\end{proof}

\begin{lemma}\label{boundary} We have
\begin{multline}(v_1,v_2,v_3,v_4)_{v_0}^0-(v_0,v_2,v_3,v_4)_{v_1}^1+(v_0,v_1,v_3,v_4)_{v_2}^2\\-(v_0,v_1,v_2,v_4)_{v_3}^3+(v_0,v_1,v_2,v_3)_{v_4}^4=0\in\widehat\Pre_E(F).
\end{multline}
We will denote the left hand side by $\partial(v_0,v_1,v_2,v_3,v_4)$.
\end{lemma}
\begin{proof} As in the proof of Lemma~\ref{boundaryi}, we can verify that the flattenings satisfy the five equations in Definition~\ref{fiveeqdef}. We leave this to the reader.
\end{proof}

If $\F$ is an ordered basis of $F^k$, we let $\F_i$ denote the $i$th basis vector. A set $S$ of ordered bases is in \emph{general position} if any set of $k$ basis vectors from $S$ is linearly independent.
Let $CF_*$ be the chain complex generated in dimension $n$ by tuples $(\F 0,\dots,\F n)$ of ordered bases of $F^3$ in general position, together with the usual boundary map. Left multiplication makes $CF_*$ into a chain complex of free $\GL(3,F)$--modules. Since $F$ is assumed to be infinite, it is easy to see that $CF_*$ is acyclic. Hence, the complexes $(CF_*)_{\SL(3,F)}$ and $(CF_*)_{\GL(3,F)}$ compute the homology groups $H_*(\SL(3,F))$ and $H_*(\GL(3,F))$, respectively.

Consider the $\SL(3,F)$--invariant map $\widehat \lambda\colon CF_3\to \widehat \Pre_E(F)$ given by sending a generator $(\F0,\F1,\F2,\F3)$ to
\begin{multline}\label{flagtoBhat}(\F0_2,\F1_1,\F2_1,\F3_1)_{\F 0_1}^0+(\F0_1,\F1_2,\F2_1,\F3_1)_{\F 1_1}^1\\+(\F0_1,\F1_1,\F2_2,\F3_1)_{\F 2_1}^2+(\F0_1,\F1_1,\F2_1,\F3_2)_{\F 3_1}^3.
\end{multline}

We will often abbreviate the notation by omitting the $\F$, and writing a subscript $\F i_1$ as $i$, e.g.~we abbreviate $(\F0_2,\F1_1,\F2_1,\F3_1)_{\F0_1}^0$ to 
$(0_2,1_1,2_1,3_1)_0^0$.

To help the reader visualize the arguments that follow, we give a geometric way of viewing the map $\widehat \lambda$:  
A generator of $CF_3$ can be thought of as a simplex $\Delta$ together with an association of an ordered basis $\F i$ to each vertex.
We may think of each of the four terms in \eqref{flagtoBhat} as a standard simplex endowed with an ideal cochain.
We mark $\Delta$ with two points on each edge and a point on each face as shown in Figure~\ref{flagfigure}. We will refer to these points as \emph{edge points} and \emph{face points} respectively. 
Each point is given uniquely by a tuple $\beta=(x_0,x_1,x_2,x_3)$ with $x_0+x_1+x_2+x_3=3$, where the coordinate $x_i$ measures the ``distance'' to the face opposite vertex $i$. For each such $\beta$, let $\beta_i$ be the ordered set $\{\F i_1,\dots,\F i_{x_i}\}$ and let $S_\beta=\beta_0\cup\beta_1\cup\beta_2\cup\beta_3$. Note that $S_\beta$ always has exactly $3$ elements. Hence, $\det(S_\beta)$ is well defined and gives a labeling of each marked point of $\Delta$. 
As an example, the edge point, closest to vertex $1$, between vertices $1$ and $2$ is labeled by $\det(\F1_1,\F1_2,\F2_1)$.

We can think of $\Delta$ as a union of four simplices $\Delta_i$, where $\Delta_i$ is the simplex spanned by the $i$th vertex of $\Delta$ and the marked points with $x_i\neq 0$. We think of the $\Delta_i$'s as being disjoint. The labelings of the marked points in $\Delta$ gives rise to cochains on $\Delta_i$, and using~\eqref{alphadets}, one can check that these are exactly the ideal cochains of the terms in \eqref{flagtoBhat}.

\begin{figure}[ht]
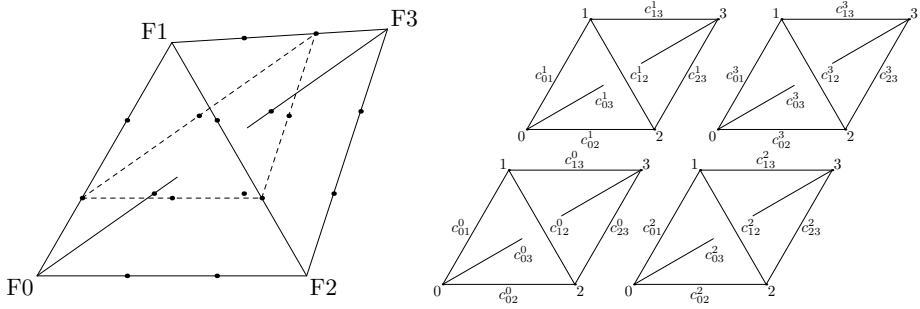

\centering
\begin{minipage}[c]{0.47\textwidth}
 \text{\hspace{-1.2cm}}\includegraphics[width=5.5cm]{K3Bhat.3}
\end{minipage}
\hspace{-1.7cm}
\begin{minipage}[c]{0.15\textwidth}
 \text{\hspace{1.12cm}}\includegraphics[width=2.8cm]{K3Bhat.5}
 \includegraphics[width=2.9cm]{K3Bhat.4}
\end{minipage}
\quad\enspace
\begin{minipage}[c]{0.2\textwidth}
  \text{\hspace{1.12cm}}\includegraphics[width=2.8cm]{K3Bhat.7}
 \includegraphics[width=2.9cm]{K3Bhat.6}
\end{minipage}
\caption{The ideal cochains on the simplices $\Delta_i$ arising from a labeling of marked points in $\Delta$. The dashed lines mark the bottom of~$\Delta_1$.}\label{flagfigure}
\end{figure}

\begin{remark}\label{flags} An ordered basis determines an affine flag, and one easily checks that $\widehat\lambda$ only depends on the underlying affine flags.
\end{remark}

If $\tau\in (CF_3)_{\SL(3,F)}$ is a cycle, we can represent $\tau$ by a $3$--cycle $K$ together with a labeling of the marked points in each of the simplices of $K$. Identified points acquire the same labeling. From the geometric description of the map $\widehat\lambda$, it follows that we can represent $\widehat\lambda(\tau)$ by a $3$--cycle $C$ with boundary together with an ideal cochain on $C$. Note that $C$ is homeomorphic to the disjoint union of the cones on the links of the $0$--cells of $K$.

\begin{lemma}\label{actiononedgeandface}
The restriction of $\widehat\lambda$ to cycles in $(CF_3)_{\SL(3,F)}$ is independent of the choice of logarithm. In fact, we may choose different logarithms for each of the marked points as long as we use the same logarithm for identified points.
\end{lemma}
\begin{proof}
For edge points this follows from Lemma~\ref{sameeltlemma}. A face point occurs in exactly $3$ of the simplices $\Delta_i$. Consider the face point opposite vertex $0$, and suppose the flattenings of $\Delta_1$, $\Delta_2$ and $\Delta_3$ are $(e,f)$, $(e',f')$ and $(e'',f'')$. If we add $1$ to the logarithm of the face point, it follows from \eqref{alphaef} that the flattenings become $(e,f+1)$, $(e'-1,f'-1)$ and $(e''+1,f'')$. By Lemma~\ref{Rformula}, this changes the element in $\widehat\Pre_E(F)$ by $\chi(e-e'+f'-f''+1)$. Keep in mind that $\chi(1)=-\chi(1)$. Using \eqref{alphaef}, we have
\begin{equation}\label{facepoint}
\begin{aligned}
e-e'+f'-f''&=(0_1,1_1,3_1)+(1_1,1_2,2_1)-(0_1,1_1,2_1)-(1_1,1_2,3_1)\\
&\enspace-\big((0_1,2_1,3_1)+(1_1,2_1,2_2)-(0_1,2_1,2_2)-(1_1,2_1,3_1)\big)\\
&\enspace+(0_1,1_1,2_1)+(2_1,3_1,3_2)-(0_1,2_1,2_2)-(1_1,2_1,3_1)\\
&\enspace-\big((0_1,1_1,3_1)+(2_1,3_1,3_2)-(0_1,2_1,3_1)-(1_1.3_1,3_2)\big)\\
&=\big((1_1,1_2,2_1)-(1_1,2_1,2_2)\big)+\big((1_1,3_1,3_2)-(1_1,1_2,3_1)\big)\\
&\enspace+\big((2_1,2_2,3_1)-(2_1,3_1,3_2)\big),
\end{aligned}
\end{equation}
where $(i_j,k_l,m_n)$ denotes $\log\det(\F i_j,\F k_l,\F m_n)$. 
Note that each term is a logarithm of one of the six edge points on the face opposite vertex $0$. By a similar calculation one can check that this holds in general, i.e.~the change in the $\widehat\Pre_E(F)$ element when adding $1$ to the logarithm of a face point, is a signed sum of logarithms of the edge points on that face (plus $\chi(1)$). The signs are shown in Figure~\ref{signs}.

\begin{figure}[ht]
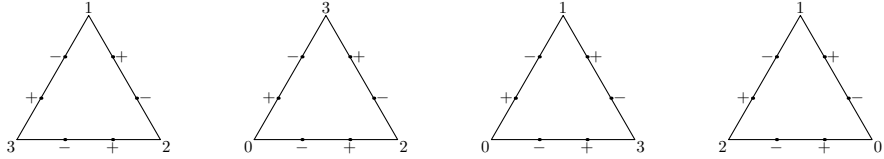

\centering
\begin{minipage}[c]{0.24\textwidth}
\includegraphics[width=2.2cm]{K3Bhat.8}
\end{minipage}
\begin{minipage}[c]{0.24\textwidth}
\includegraphics[width=2.2cm]{K3Bhat.9}
\end{minipage}
\begin{minipage}[c]{0.24\textwidth}
\includegraphics[width=2.2cm]{K3Bhat.10}
\end{minipage}
\begin{minipage}[c]{0.24\textwidth}
\includegraphics[width=2.2cm]{K3Bhat.11}
\end{minipage}
\caption{Change in the $\widehat\Pre_E(F)$ element when adding $1$ to the logarithm of a face point. There is a contribution for each edge point on the given face.}\label{signs}
\end{figure}
In a cycle, each face point lies in exactly two simplices, and since the face pairings preserve orderings, it follows from Figure~\ref{signs} that the changes in the element in $\widehat\Pre_E(F)$, resulting from adding $1$ to the logarithm, appear with opposite signs. 
\end{proof}

\begin{lemma}\label{boundarytozero} $\widehat \lambda$ takes boundaries in $CF_3$ to $0\in \widehat \Pre_E(F)$. 
\end{lemma}
\begin{proof}
Using \eqref{flagtoBhat}, we see that $\widehat\lambda(\partial(\F0,\dots,\F4))\in \widehat \Pre_E(F)$ equals
{\small{
\begin{gather*}
\begin{aligned}
&+(1_2,2_1,3_1,4_1)_1^0-(0_2,2_1,3_1,4_1)_0^0+(0_2,1_1,3_1,4_1)_0^0-(0_2,1_1,2_1,4_1)_0^0+(0_2,1_1,2_1,3_1)_0^0\\
&+(1_1,2_2,3_1,4_1)_2^1-(0_1,2_2,3_1,4_1)_2^1+(0_1,1_2,3_1,4_1)_1^1-(0_1,1_2,2_1,4_1)_1^1+(0_1,1_2,2_1,3_1)_1^1\\
&+(1_1,2_1,3_2,4_1)_3^2-(0_1,2_1,3_2,4_1)_3^2+(0_1,1_1,3_2,4_1)_3^2-(0_1,1_1,2_2,4_1)_2^2+(0_1,1_1,2_2,3_1)_2^2\\
&+(1_1,2_1,3_1,4_2)_4^3-(0_1,2_1,3_1,4_2)_4^3+(0_1,1_1,3_1,4_2)_4^3-(0_1,1_1,2_1,4_2)_4^3+(0_1,1_1,2_1,3_2)_3^3.
\end{aligned}
\end{gather*}}}
Using Lemma~\ref{boundaryi}, this simplifies to 
\begin{multline*}
-(1_1,2_1,3_1,4_1)_0^0+(0_1,2_1,3_1,4_1)_1^1-(0_1,1_1,3_1,4_1)_2^2\\+(0_1,1_1,2_1,4_1)_3^3-(0_1,1_1,2_1,3_1)_4^4,
\end{multline*}
which by Lemma~\ref{boundary} is $0\in \widehat \Pre_E(F)$.
\end{proof}

We thus obtain an induced map $\widehat\lambda\colon H_3(\SL(3,F))\to \widehat\Pre_E(F)$. 

\begin{lemma} The image of $\widehat\lambda$ is in $\widehat\B_E(F)$.
\end{lemma}
\begin{proof} 
Consider the sequence of maps $J_n\colon CF_n\to I_n$ given by
\begin{equation}
(\F0,\dots,\F n)\mapsto \sum_{i=0}^n(\F0_1,\dots,\F i_2,\dots,\F n_1)^i_{\F i_1}.
\end{equation}
Note that $J$ is \emph{not} a chain map. By definition, $\widehat\lambda\colon CF_3\to \widehat\Pre_E(F)$ is equal to $\widehat\sigma\circ J_3$, where $\widehat\sigma\colon I_3\to\widehat\Pre_E(F)$ is the map given by \eqref{cijflat}. Consider the diagram
\begin{equation}\label{mudiag}
\cxymatrix{{CF_3\ar[r]^-{J_3}\ar[d]^-\partial&I_3\ar[r]^-{\widehat\sigma}\ar[d]^-\partial&{\widehat\Pre_E(F)}\ar[d]^-{\widehat\nu}\\
CF_2\ar[r]^-{J_2}&I_2\ar[r]^-\mu&\wedge^2(E).}}
\end{equation}
By Lemma~\ref{I3Bhat} the right square is commutative.
Using the usual notational abbreviations, i.e. omitting the $\F$, and shortening subscripts to $i$, a direct computation shows that
\begin{multline}\delta:=(\partial J_3-J_2\partial)(\F0,\F 1,\F 2,\F 3)=\\(1_1,2_1,3_1)_0^0-(0_1,2_1,3_1)_1^1+(0_1,1_1,3_1)_2^2-(0_1,1_1,2_1)_3^3.\end{multline}
One easily checks that $\mu$ takes $\delta$ to $0\in \wedge^2(E)$, and the result follows.
\end{proof}

\begin{lemma}\label{SL2SL3}
The restriction of $\widehat\lambda$ to $H_3(\SL(2,F))$ agrees with the map from the previous section.
\end{lemma}
\begin{proof}
We consider $F^2$ as a subspace of $F^3$ using the inclusion $(x,y)\mapsto (0,x,y)$. 
Let $p\colon F^3\to F^2$ be the natural projection, and let $D_*$ be the subcomplex of $CF_*$ consisting of tuples $(\F0,\dots,\F n)$ such that $(p\F0_1,\dots ,p\F n_1)\in C_n(F^2)$.
Note that $D_*$ is an acyclic $\SL(2,F)$--complex, where $\SL(2,F)$ is regarded as a subgroup of $\SL(3,F)$ in the natural way. 
Consider the $\GL(2,F)$--equivariant map
\begin{equation}
\Psi\colon D_*\to C_*(F^2),\quad (\F0,\dots,\F n)\mapsto (p\F0_1,\dots,p\F n_1).
\end{equation}
Let $\widehat\tau$ denote the map $C_3(F^2)\to \widehat\Pre_E(F)$ from Section~\ref{idealcochainSection}. We wish to prove that $\widehat\tau\circ\Psi$ and $\widehat\lambda$ differ by a coboundary. Note that $(0_1,1_1,2_1,3_1)^0_w=\widehat\tau\circ\Psi(\F0,\F1,\F2,\F3)$.

By definition, $\widehat\lambda$ takes  $(\F0,\F1,\F2,\F3)\in D_3$ to
\begin{equation*}(0_2,1_1,2_1,3_1)_0^0+(0_1,1_2,2_1,3_1)_1^1+(0_1,1_1,2_2,3_1)_2^2+(0_1,1_1,2_1,3_2)_3^3.
\end{equation*}
We may subtract boundaries without effecting the image in $\widehat\Pre_E(F)$, and after subtracting
\begin{multline*}
\partial^1(w,0_2,1_1,2_1,3_1)_0+\partial^2(w,0_1,1_2,2_1,3_1)_1\\+\partial^3(w,0_1,1_1,2_2,3_1)_2+\partial^4(w,0_1,1_1,2_1,3_2)_3,
\end{multline*}
the remaining terms become
\begin{gather*}
\begin{aligned}
&(w,1_1,2_1,3_1)_0^1-(w,0_2,2_1,3_1)^1_0+(w,0_2,1_1,3_1)_0^1-(w,0_2,1_1,2_1)_1^1\\
&(w,1_2,2_1,3_1)^1_1-(w,0_1,2_1,3_1)_1^2+(w,0_1,1_2,3_1)^2_1-(w,0_1,1_2,2_1)^2_1\\
&(w,1_1,2_2,3_1)^2_2-(w,0_1,2_2,3_1)^2_2+(w,0_1,1_1,3_1)_2^3-(w,0_1,1_1,2_2)^3_2\\
&(w,1_1,2_1,3_2)^3_3-(w,0_1,1_1,3_2)_3^3+(w,0_1,1_1,3_2)_3^3-(w,0_1,1_1,2_1)^4_3.
\end{aligned}
\end{gather*}

By Lemma~\ref{boundary} the diagonal terms sum to $(0_1,1_1,2_1,3_1)^0_w=\widehat\tau\circ\Psi(\F0,\F1,\F2,\F3)$.
If we define $\phi\colon D_2\to \widehat\Pre_E(F)$ by
\[\phi(\F0,\F1,\F2)=(w,0_2,1_1,2_1)^1_0+(w,0_1,1_2,2_1)_1^2+(w,0_1,1_1,2_2)_2^3,\]
the remaing terms are easily seen to equal $\phi\circ\partial (\F0,\F1,\F2,\F3)$. 
Hence, $\widehat\tau\circ\Psi$ and $\widehat\lambda$ differ by a coboundary as desired.
\end{proof}

\begin{remark}\label{GL2GL3} As in Remark~\ref{GL2}, the map $\pi\circ\widehat\lambda\colon CF_3\to\Pre(F)$ is $\GL(3,F)$ invariant, and induces a map $H_3(\GL(3,F))\to\B(F)$, which factors through the homology of the complex $PCF_*$ of projective bases of $F^3$. The proof of Lemma~\ref{SL2SL3} shows that the map $H_3(\GL(3,F))\to\B(F)$ agrees with the map in Remark~\ref{GL2}, and that the map $H_3(\GL(2,F))\to\B(F)$ lifts to $\widehat\B_E(F)$ via $H_3(\SL(3,F))$ and the stabilization map $\GL(2,F)\to\SL(3,F)$.
\end{remark}

\subsubsection{Extension to $H_3(\GL(n,F))$} In Garoufalidis--Thurston--Zickert~\cite{GaroufalidisThurstonZickert} we construct maps $H_3(\SL(n,F))\to\widehat\B_E(F)$ commuting with the stabilization maps. These maps are induced by an $\SL(n,F)$--invariant map $\widehat\lambda\colon CF^n_3\to \widehat\Pre_E(F)$, where $CF^n_*$ is the complex of ordered bases of $F^n$ (or affine flags, c.f.~Remark~\ref{flags}). Remark~\ref{GL2GL3} generalizes, i.e.~the maps
\begin{equation}
H_3(\GL(n,F))\to H_3(\SL(n+1,F))\to\widehat\B_E(F)
\end{equation} 
commute with stabilization. Hence, by~\eqref{MaponK}, we obtain a map $K_3(F)\to\widehat\B_E(F)$.
If $F$ is free (and infinite), $\widehat\lambda$ commutes with the maps $\Psi_{xy}$ from Proposition~\ref{extmap}, so $\widehat\lambda$ induces a map $K_3(F)\to\widehat\B(F)$. The map $\widehat\lambda$ commutes with Galois actions, and respects the maps induced by embeddings in $\C$. This implies that the regulators \eqref{regonK3} and \eqref{regonBhat} agree.

\subsection{Step two: $K_3^M(F)$ maps to zero}\label{steptwo}

From now on, we assume that $F$ is a free field admitting an embedding in $\C$. This is used in Proposition~\ref{compositioneq0} but not in Lemma~\ref{agreelemma}.

\begin{lemma}\label{agreelemma} The composition 
\[\xymatrix{{H_3(\GL(3,F))}\ar[r]^-{\widehat\lambda}&{\widehat\B_E(F)}\ar[r]^-\pi&{\B(F)}}\]
agrees with the map constructed by Suslin~\cite[Section 3]{Suslin}.
\end{lemma}
\begin{proof}
By a result of Suslin~\cite{Suslin1046}, $H_3(\GL(3,F))$ is generated by $H_3(\GL(2,F))$ and $H_3(T)$. By Remark~\ref{GL2} the two maps on $H_3(\GL(2,F))$. By Remark~\ref{GL2GL3} the map $\pi\circ\widehat\lambda$ factors through the complex $PCF_*$ of projective bases. Since $T$ acts trivially on projective bases, $PCF_0\to \Z$ has a $T$--equivariant section, and it follows that $\pi\circ\widehat\lambda$ is $0$ on $H_3(T)$. By Suslin~\cite[Proposition~3.1]{Suslin}, this also holds for Suslin's map. Hence, the two maps agree. 
\end{proof}

\begin{prop}\label{compositioneq0}
The composition $K_3^M(F)\to K_3(F)\to \widehat\B(F)$ is $0$.
\end{prop}
\begin{proof}
Let $\sigma\colon F\to \C$ be an embedding and let $\sigma_*\colon\widehat\B(F)\to\widehat\B(\C)$ be the induced map. 
By Lemma~\ref{agreelemma}, the image of $K_3^M(F)$ in $\widehat\B(F)$ is in $\widetilde{\mu_F}$. 
By Proposition~\ref{injectiveontor}, $\sigma_*$ maps $\widetilde{\mu_F}$ injectively to $\mu_\C$, and since the regulator $R$ is injective on $\mu_\C$, it is enough to prove that the composition
\[\xymatrix{{K_3^M(F)}\ar[r]&{\widetilde{\mu_F}}\ar[r]&{\mu_\C}\ar[r]^-R&{\C/4\pi^2\Z}}\]
is zero. Since this factors through $K_3^M(\C)$, the result follows from Theorem~\ref{Sahthm} and Theorem~\ref{H3eqBhat}.\end{proof}

\subsection{Step three: A five lemma argument}\label{stepthree}

\begin{lemma}\label{relationtoSuslin}
The exact sequences \eqref{Suslinexact} and \eqref{exactforfreefields} fit together in a diagram
\[\xymatrix{0\ar[r]&{\widetilde{\mu_F}}\ar@{=}[d]\ar[r]&{K_3^{\ind}(F)}\ar[r]\ar[d]^-{\widehat\lambda}&{\B(F)}\ar@{=}[d]\ar[r]&0
\\0\ar[r]&{\widetilde{\mu_F}}\ar[r]&{\widehat\B(F)}\ar[r]&{\B(F)}\ar[r]&0.}\]
\end{lemma}
\begin{proof}
Commutativity of the right square follows from Lemma~\ref{agreelemma}.
To prove commutativity of the left square, we proceed as in the proof of Lemma~\ref{compositioneq0}.
Since $\sigma_*$ is injective on $\widetilde{\mu_F}$, it is enough to prove the corresponding result with $F$ replaced by $\C$. The result now follows from Theorem~\ref{Sahthm} and Theorem~\ref{H3eqBhat}. 
\end{proof}

The theorem below summarizes our results.
\begin{thm}\label{generalmainthm} Let $F$ be a free field admitting an embedding in $\C$.
There is a natural isomorphism 
\[\widehat\lambda\colon K_3^{\ind}(F)\cong \widehat\B(F)\]
commuting with Galois actions.\qed
\end{thm}  

If $F\subset E$ is a field extension, the natural map $K_3^{\ind}(F)\to K_3^{\ind}(E)$ is an inclusion. Furthermore, if $F\subset E$ is Galois, we have $K_3^{\ind}(E)^{\Gal(E,F)}=K_3^{\ind}(F)$. This property is called Galois descent. We refer to Merkurjev--Suslin~\cite{MerkurjevSuslin} for proofs.

\begin{cor}\label{Bhatinjective}
For any free subfield $F$ of $\C$, the map $\widehat\B(F)\to\widehat\B(\C)$ induced by inclusion is injective. \qed
\end{cor}

\begin{cor}\label{Galoisdescent}
The extended Bloch group of a number field satisfies Galois descent.\qed
\end{cor}

\section{Torsion in the extended Bloch group}\label{TorinBhat}
In this section we give a concrete description of the torsion in $\widehat\B(F)$. We start by reviewing some elementary properties of homology of cyclic groups.

\begin{prop}\label{Hofcyclic} Let $G$ be a cyclic group of order $n$ generated by an element $g\in G$.
The homology group $H_3(G)$ is cyclic of order $n$ and is generated by the cycle
\[\sum_{k=1}^n\langle g\vert g^k\vert g\rangle.\]
We may thus identify $G$ with $H_3(G)$.\qed
\end{prop}
We refer to Parry--Sah~\cite[Proposition 3.25]{ParrySah} for an algebraic proof, and to Neumann~\cite{Neumann} for a geometric proof using the lens space $L(n,1)$.

Let $p$ be a prime number. As explained in the introduction, $K_3^{\text{ind}}(F)_p$ has order $p^{\nu_p}$ for $p$ odd and $2p^{\nu_p}$ for $p=2$, where $\nu_p=\text{max}\{\nu\mid \xi_{p^\nu}+\xi_{p^\nu}^{-1}\in F\}$. To construct the torsion in $\widehat\B(F)$, it is thus enough to exhibit elements in $\widehat\B(F)_p$ of order $p^{\nu_p}$ for $p$ odd and $2p^{\nu_p}$ for $p=2$. 

Let $n=p^{\nu_p}$, and let $x$ be a primitive $n$th root of unity. Consider the matrices
\begin{equation}\label{matrices}g=\begin{pmatrix}x+x^{-1}&-1\\1&0\end{pmatrix},\quad \mu=\begin{pmatrix}x&0\\0&x^{-1}\end{pmatrix},\quad X=\begin{pmatrix}x&1\\1&x\end{pmatrix}.\end{equation}
Note that $g\in \SL(2,F)$ and that $g=X\mu X^{-1}$. Hence, $g$ generates a cyclic subgroup of $\SL(2,F)$ of order $n$. 
Let $[g]\in H_3(\SL(2,F))$ denote the homology class of the cycle $\sum_{k=1}^n\langle g\vert g^k\vert g\rangle$.

\begin{lemma}\label{torsionidea} The element $\widehat\lambda([g])\in \widehat\B(F)$ has order $n=p^{\nu_p}$. 
\end{lemma}
\begin{proof} 
It follows from Proposition~\ref{Hofcyclic} that $\widehat\lambda([g])$ has order at most $n$. If we fix an embedding of $F(x)$ in $\C$, we can view $g$ and $\mu$ as elements in $\SL(2,\C)$. Since $g$ and $\mu$ are conjugate in $\SL(2,\C)$, it follows from Theorem~\ref{torsiondiag} that $\widehat\lambda([g])$ has order at least $n$. Hence, $\widehat\lambda([g])$ has order $n$.
\end{proof}

\begin{cor} For $p$ odd, $\widehat\B(F)_p$ is generated by $\widehat\lambda([g])$.\qed
\end{cor}

\subsection{Explicit computations}
We now give an explicit expression for $\widehat\lambda([g])$. Assume for now that $p$ is odd.  

For any $h_1$ and $h_2$ in $\SL(2,F)$, there is a homogeneous representative of $[g]$ of the form
\begin{equation}\label{hrepresentative}\sum_{i=1}^n(h_1,gh_1,g^kh_2,g^{k+1}h_2),\end{equation}
see e.g.~the example in Neumann~\cite[Section~12]{Neumann}. Using \eqref{giflat}, we see that $\widehat\lambda$ takes a term $(h_1,gh_1,g^kh_2,g^{k+1}h_2)$ to a flattening $(e_k,f_k)$, with
\begin{equation}\label{ekfk}
\begin{gathered}
e_k=\log(\det(v_1,g^{k+1}v_2))+\log(\det(v_1,g^{k-1}v_2))-2\log(\det(v_1,g^kv_2))\\
f_k=\log(\det(v_1,gv_1))+\log(\det(v_2,gv_2))-2\log(\det(v_1,g^kv_2)),
\end{gathered}
\end{equation}
where $v_1=h_1\bigl(\begin{smallmatrix}1\\0\end{smallmatrix}\bigr)$ and $v_2=h_2\bigl(\begin{smallmatrix}1\\0\end{smallmatrix}\bigr)$. It follows that $\widehat\lambda([g])=\sum_{i=1}^n(e_k,f_k)\in\widehat\B(F)$.

Since the cycles \eqref{hrepresentative} all represent $[g]$, we may choose $v_1$ and $v_2$ as we please (as long as  
the vectors $v_1,gv_1,g^kv_2,g^{k+1}v_2$ in $F^2$ are in general position) without effecting the element in $\widehat\B(F)$. 
If we let $v_1=\bigl(\begin{smallmatrix}1\\-1\end{smallmatrix}\bigr)$ and $v_2=\bigl(\begin{smallmatrix}1\\1\end{smallmatrix}\bigr)$, we have
\begin{equation}\label{poddcomp}
\begin{aligned}
\det(v_1,g^kv_2)&=\det(v_1,X\mu^k X^{-1}v_2)\\
&=(x^2-1)\det(X^{-1}v_1,\mu^kX^{-1}v_2)\\
&=\frac{1}{x^2-1}\det\bigl(\bigl(\begin{smallmatrix}x+1\\-x-1\end{smallmatrix}\bigr),\mu^k\bigl(\begin{smallmatrix}x-1\\x-1\end{smallmatrix}\bigr)\bigr)\\
&=\det\bigl(\bigl(\begin{smallmatrix}1\\-1\end{smallmatrix}\bigr),\bigl(\begin{smallmatrix}x^k&0\\0&x^{-k}\end{smallmatrix}\bigr)\bigl(\begin{smallmatrix}1\\1\end{smallmatrix}\bigr)\bigr)\\
&=x^k+x^{-k}.
\end{aligned}
\end{equation}
Letting $z_k$ denote the corresponding cross-ratio of $(e_k,f_k)$, it follows from \eqref{ekfk} that 
\[z_k=\frac{(x^{k+1}+x^{-k-1})(x^{k-1}+x^{-k+1})}{(x^k+x^{-k})^2}.\]
Since $p$ is assumed to be odd, $z_k\in F\setminus\{0,1\}$.
This proves Theorem~\ref{torsioninB} for $p$ odd.

Suppose $p=2$. By computations similar to \eqref{poddcomp} using $v_1=\bigl(\begin{smallmatrix}1\\0\end{smallmatrix}\bigr)$ and $v_2=\bigl(\begin{smallmatrix}1\\-1\end{smallmatrix}\bigr)$ we obtain
\begin{equation}\label{p2}\det(v_1,g^kv_2)=\frac{x^k-x^{-k+1}}{x-1}, \quad \det(v_1,gv_1)=1,\quad \det(v_2,gv_2)=2+x+x^{-1}.\end{equation}
We wish to prove that $\widehat\lambda([g])\in \widehat\B(F)$ is $2$--divisible.

Let $c_k=\det(v_1,g^kv_2)$ and let $\tilde c_k=\log(c_k)$. Also, let $a=2+x+x^{-1}$ and let $\tilde a=\log(a)$. 
By \eqref{p2}, we see that $c_k=c_{n-k+1}$ and $c_k=-c_{k+n/2}$. 
By \eqref{ekfk}, 
\[(e_k,f_k)=(\tilde c_{k+1}+\tilde c_{k-1}-2\tilde c_k,\tilde a-2\tilde c_k).\]
We may choose different logarithms for each $k$ without effecting the element $\widehat\lambda([g])=\sum_{i=1}^n(e_k,f_k)$. We will choose them such that $\tilde c_k-\tilde c_{n-k+1}$ and $\tilde c_{k+n/2}-\tilde c_k$ are independent of $k$ and such that $2\tilde c_k=2\tilde c_{k+n/2}$. With these particular choices, it is easy to see that $\widehat\lambda([g])$ is $2$--divisible. Indeed, $\widehat\lambda([g])=2Q$, where 
\begin{equation}
Q=\sum_{i=1}^{n/2}(e_k,f_k)\in\widehat\Pre(F). 
\end{equation}
We now only need to prove that $Q$ is in $\widehat\B(F)$. This follows from the computation
\begin{gather*}
\begin{aligned}
\widehat\nu(Q)&=\sum\nolimits_{k=1}^{n/2}(\tilde c_{k+1}+\tilde c_{k-1}-2\tilde c_k)\wedge(\tilde a-2\tilde c_k)\\
&=\sum\nolimits_{k=1}^{n/2}2\tilde c_k\wedge(\tilde c_{k+1}+\tilde c_{k-1})-\sum\nolimits_{k=1}^{n/2}2\tilde c_k\wedge \tilde a + \sum\nolimits_{k=1}^{n/2}(\tilde c_{k+1}+\tilde c_{k-1})\wedge \tilde a\\
&=\sum\nolimits_{k=1}^{n/2}(2\tilde c_k\wedge \tilde c_{k+1}-2\tilde c_{k-1}\wedge \tilde c_k)+
\sum\nolimits_{k=1}^{n/2}\big((\tilde c_{k+1}-\tilde c_k)-(\tilde c_k-\tilde c_{k-1})\big)\wedge \tilde a\\
&=2\tilde c_{n/2}\wedge\tilde c_{n/2+1}-2\tilde c_0\wedge\tilde c_1+(\tilde c_{n/2+1}-\tilde c_1)\wedge \tilde a-(\tilde c_{n/2}-\tilde c_0)\wedge\tilde a\\
&=0\in \wedge^2(E).
\end{aligned}
\end{gather*}

Since \[z_k=\frac{c_{k+1}c_{k-1}}{c_k^2}=\frac{(x^{k+1}-x^{-k})(x^{k-1}-x^{-k+2})}{(x^k-x^{-k+1})^2}\in F,\]
this proves Theorem~\ref{torsioninB} for $p=2$. 

We give some examples below. The computational details are left to the reader.

\begin{ex} For any number field $F$, which does not contain a $3$rd root of unity, the element $2[-2]+[\frac{1}{4}]\in\B(F)$ has order~$3$.
\end{ex}
\begin{ex} Let $F=\Q(\sqrt 2)$. 
Doing the above computations, we obtain that 
\[Q=[\sqrt 2-1;0,0]+[\sqrt 2-1;0,-2]+[-\sqrt 2-1;0,0]+[-\sqrt 2-1;-2,-2]\in \widehat\B(F).\]
It follows that the element $\beta_2=2[\sqrt 2-1]+2[-\sqrt 2-1]\in \B(F)$ has order $4$ and generates $\B(F)_2$. 
Note that $\beta_2$ is \emph{not} $2$--divisible. Applying the regulator~\eqref{regulator}, we get $R(Q)=\pi^2/4\in \C/4\pi^2$, which has order $16=2^{\nu_2+1}$ as expected.
\end{ex}

\begin{remark} The order of the torsion in $K_3^{\text{ind}}(F)$ is always divible by $24$. A particular generator of the $24$--torsion in $\widehat\B(F)$ is given by the element $(e,f)+(f,e)$ from Lemma~\ref{N7.3}. We omit the proof of this.
\end{remark}

\section{Hyperbolic $3$--manifolds}\label{hypthm}
Let $M$ be a complete, oriented, hyperbolic $3$--manifold with finite volume, and let $K$ and $k$ denote the trace field and invariant trace field of $M$. By a result of Goncharov~\cite[Theorem~1.1]{Goncharov}, $M$ defines an element $[M]$ in $K_3^{\ind}(\overline\Q)\otimes\Q$, which equals the Bloch invariant of $M$ (see e.g.~Neumann--Yang~\cite{NeumannYang}) under the isomorphism 
\[K_3^{\ind}(\overline\Q)\otimes\Q\cong \B(\overline\Q).\] 

Recall that a spin structure on $M$ is equivalent to a lift of the geometric representation to $\SL(2,\C)$, and that the set of spin structures is an affine space over $H^1(M;\Z/2\Z)$. We thank Walter Neumann for assistance with the proof of the result below.
\begin{thm}
Suppose $M$ is closed. A spin structure $\rho$ on $M$ determines a fundamental class $[M_\rho]$ in $K_3^{\ind}(K)$ lifting the Bloch invariant. For each $\alpha\in H^1(M;\Z/2\Z)$, the element $[M_\rho]-[M_{\alpha\rho}]$ is two-torsion, which is trivial if and only if the induced map $B\alpha_*\colon H_3(M)\to H_3(B(\Z/2\Z))=\Z/2\Z$ is trivial. In particular, $2[M_\rho]$ is independent of $\rho$. Moreover, $2[M_\rho]$ is in $K_3^{\ind}(k)$.
\end{thm} 
\begin{proof}
By Reid--Maclachlan \cite[Corollary 3.2.4]{ReidMaclachlan}, we may assume that $\rho$ has image in $\SL(2,K(\lambda))$, where $\lambda$ is an algebraic element of degree at most $2$ over $K$. Endow $M$ with the structure of a closed $3$--cycle, 
and fix an $\SL(2,K(\lambda))$--cocycle $\alpha$ on $M$ representing the fundamental class $[\rho]$ of $\rho$ in $H_3(\SL(2,K(\lambda)))$; see e.g.~Zickert~\cite[Section~5]{Zickert}. Let $[M_\rho]=\widehat\lambda([\rho])$. Then $[M_\rho]$, is given by the ideal cochain $c$ on $M$ defined by $\alpha$ using \eqref{giflat} and \eqref{defofGamma}. 
If $\lambda$ has degree $1$, $[M_\rho]$ is obviously in $\widehat\B(K)$. If $\lambda$ has degree $2$, the non-trivial element in $\Gal(K(\lambda),K)$ preserves traces of $\rho$ ($K$ is the trace field), and therefore takes $\rho$ to a representation which is conjugate over $\C$. It follows that the image of $[M_\rho]$ in $\widehat\B(\C)$ is invariant under $\Gal(K(\lambda),K)$, so by the Corollaries \ref{Bhatinjective} and \ref{Galoisdescent}, $[M_\rho]$ is in $\widehat\B(K)$.

The second statement follows from Theorem~\ref{Z2action}, so we now only need to prove that $2[M_\rho]$ is in $\widehat\B(k)$. By Neumann--Reid~\cite[Theorem~2.1]{NeumannReid}, $K$ is Galois over $k$. Let $\sigma\in\Gal(K,k)$. Since $k$ is the field of squares of traces of $\rho$, it follows that $\sigma\rho$ as a representation in $\PSL(2,\C)$ is conjugate to the geometric representation. After a conjugation (which does not change the fundamental class), we may thus assume that $\rho$ and $\sigma\rho$ are equal as representations in $\PSL(2,\C)$. Hence, $c$ and $\sigma(c)$ differ by a $\Z/2\Z$--cocycle, so by Theorem~\ref{Z2action}, $2\sigma([M_\rho])=2[M_\rho]\in\widehat\B(\C)$. As above, this implies that $2[M_\rho]$ is in $\widehat\B(k)$. 
\end{proof}

\subsection{Cusped manifolds}\label{cuspsection}
If $M$ has cusps, Reid--Maclachlan~\cite{ReidMaclachlan} shows that the geometric representation has image in $\PSL(2,K)$.
It thus follows from Theorem~\ref{Zickertthm} that $M$ has a fundamental class $[M]\in\widehat\B(K)_{\PSL}$. 
Neumann--Yang~\cite{NeumannYang} show that the Bloch invariant of $M$ is always in $\B(k)$, but they define $\B(k)$ as the kernel of $z\mapsto 2z\wedge (1-z)$. With our definition, only $2[M]$ is in $\B(k)$. An explicit example whith $[M]\notin\B(k)$ is given by the manifold $m009$ in the SnapPea census. Similarly, only $2[M]$ is in $\widehat\B(k)_{\PSL}$. Using remark~\ref{PSLBhatlift} one checks that $2[M]$ always lifts to $\widehat\B(k)$, and by Lemma~\ref{SLPSL}, $8[M]$ lifts canonically. We do not believe that a canonical lift of $2[M]$ is possible, so this result is likely to be optimal.

\subsubsection{Knot complements}
If $M$ is a knot complement, Reid--Machlachlan~\cite[Corollary~4.2.2]{ReidMaclachlan} implies that $K=k$. The obstruction to a lift of $[M]\in\widehat\B(k)_{\PSL}$ to $\widehat\B(k)$ is a $\Z/2\Z$--valued knot invariant, which by Remark~\ref{PSLBhatlift} is explicitly computable. For example, $[M]$ lifts for the figure $8$ knot complement and the $5_2$ knot complement, but not for the $6_1$ knot complement. Since the significance of this invariant is unclear at this moment, we spare the reader for the computations.

\bibliographystyle{plain}
\bibliography{mybib}

\def\cprime{$'$}
\begin{thebibliography}{10}

\bibitem{DupontSah}
Johan~L. Dupont and Chih~Han Sah.
\newblock Scissors congruences. {II}.
\newblock {\em J. Pure Appl. Algebra}, 25(2):159--195, 1982.

\bibitem{DupontZickert}
Johan~L. Dupont and Christian~K. Zickert.
\newblock A dilogarithmic formula for the {C}heeger-{C}hern-{S}imons class.
\newblock {\em Geom. Topol.}, 10:1347--1372 (electronic), 2006.

\bibitem{FockGoncharov}
Vladimir Fock and Alexander Goncharov.
\newblock Moduli spaces of local systems and higher {T}eichm\"uller theory.
\newblock {\em Publ. Math. Inst. Hautes \'Etudes Sci.}, (103):1--211, 2006.

\bibitem{GaroufalidisThurstonZickert}
Stavros Garoufalidis, Dylan Thurston, and Christian~K. Zickert.
\newblock {\em In preparation}.

\bibitem{GZ}
Sebastian Goette and Christian Zickert.
\newblock The extended {B}loch group and the {C}heeger-{C}hern-{S}imons class.
\newblock {\em Geom. Topol.}, 11:1623--1635 (electronic), 2007.

\bibitem{Goncharov}
Alexander Goncharov.
\newblock Volumes of hyperbolic manifolds and mixed {T}ate motives.
\newblock {\em J. Amer. Math. Soc.}, 12(2):569--618, 1999.

\bibitem{Igusa}
Kiyoshi Igusa.
\newblock The {B}orel regulator map on pictures. {I}. {A} dilogarithm formula.
\newblock {\em $K$-Theory}, 7(3):201--224, 1993.

\bibitem{ReidMaclachlan}
Colin Maclachlan and Alan~W. Reid.
\newblock {\em The arithmetic of hyperbolic 3-manifolds}, volume 219 of {\em
  Graduate Texts in Mathematics}.
\newblock Springer-Verlag, New York, 2003.

\bibitem{May}
Warren May.
\newblock Fields with free multiplicity groups modulo torsion.
\newblock {\em Rocky Mountain J. Math.}, 10(3):599--604, 1980.

\bibitem{MerkurjevSuslin}
A.~S. Merkur{\cprime}ev and A.~A. Suslin.
\newblock The group {$K\sb 3$} for a field.
\newblock {\em Izv. Akad. Nauk SSSR Ser. Mat.}, 54(3):522--545, 1990.

\bibitem{Nahm}
Werner Nahm.
\newblock Conformal field theory and torsion elements of the {B}loch group.
\newblock In {\em Frontiers in number theory, physics, and geometry. {II}},
  pages 67--132. Springer, Berlin, 2007.

\bibitem{Neumann}
Walter~D. Neumann.
\newblock Extended {B}loch group and the {C}heeger-{C}hern-{S}imons class.
\newblock {\em Geom. Topol.}, 8:413--474 (electronic), 2004.

\bibitem{NeumannReid}
Walter~D. Neumann and Alan~W. Reid.
\newblock Arithmetic of hyperbolic manifolds.
\newblock In {\em Topology '90 ({C}olumbus, {OH}, 1990)}, volume~1 of {\em Ohio
  State Univ. Math. Res. Inst. Publ.}, pages 273--310. de Gruyter, Berlin,
  1992.

\bibitem{NeumannYang}
Walter~D. Neumann and Jun Yang.
\newblock Bloch invariants of hyperbolic {$3$}-manifolds.
\newblock {\em Duke Math. J.}, 96(1):29--59, 1999.

\bibitem{ParrySah}
Walter Parry and Chih-Han Sah.
\newblock Third homology of {${\rm SL}(2,\,{\bf R})$} made discrete.
\newblock {\em J. Pure Appl. Algebra}, 30(2):181--209, 1983.

\bibitem{Sah}
Chih-Han Sah.
\newblock Homology of classical {L}ie groups made discrete. {III}.
\newblock {\em J. Pure Appl. Algebra}, 56(3):269--312, 1989.

\bibitem{Suslin1046}
A.~A. Suslin.
\newblock Homology of {${\rm GL}\sb{n}$}, characteristic classes and {M}ilnor
  {$K$}-theory.
\newblock In {\em Algebraic {$K$}-theory, number theory, geometry and analysis
  ({B}ielefeld, 1982)}, volume 1046 of {\em Lecture Notes in Math.}, pages
  357--375. Springer, Berlin, 1984.

\bibitem{Suslin}
A.~A. Suslin.
\newblock {$K\sb 3$} of a field, and the {B}loch group.
\newblock {\em Trudy Mat. Inst. Steklov.}, 183:180--199, 229, 1990.
\newblock Translated in Proc.\ Steklov Inst.\ Math.\ {\bf 1991}, no.\ 4,
  217--239, Galois theory, rings, algebraic groups and their applications
  (Russian).

\bibitem{Weibel}
Charles Weibel.
\newblock Algebraic {$K$}-theory of rings of integers in local and global
  fields.
\newblock In {\em Handbook of {$K$}-theory. {V}ol. 1, 2}, pages 139--190.
  Springer, Berlin, 2005.

\bibitem{Zagier}
Don Zagier.
\newblock The dilogarithm function.
\newblock In {\em Frontiers in number theory, physics, and geometry. {II}},
  pages 3--65. Springer, Berlin, 2007.

\bibitem{Zickert}
Christian~K. Zickert.
\newblock The volume and {C}hern-{S}imons invariant of a representation.
\newblock {\em Duke Math. J.}, 150(3):489--532, 2009.

\end{thebibliography}

\end{document}